\documentclass{article}

\usepackage{amsthm}

\usepackage{amsmath,amssymb,mathrsfs}
\usepackage{hyperref,xcolor}
\usepackage{changepage}
\usepackage{enumitem} 
\usepackage{graphicx}
\usepackage{tikz}
\usepackage[utf8]{inputenc}
\usepackage[T1]{fontenc}
\usepackage{hyperref}
\usepackage{comment}

\allowdisplaybreaks

\numberwithin{equation}{section}

\newtheorem{myDefn}{Definition}[section]

\newtheorem{myProp}[myDefn]{Proposition}

\newtheorem{myRem}[myDefn]{Remark}

\newtheorem{myExa}[myDefn]{Example}

\newtheorem{myLem}[myDefn]{Lemma}

\newtheorem{myCor}[myDefn]{Corollary}

\newtheorem{myTheorem}[myDefn]{Theorem}

\DeclareMathOperator*{\argmin}{argmin}
\DeclareMathOperator*{\minn}{minimize}

\def\nn{\mathrm{n}}
\def\tt{\mathrm{\tau}}

\usepackage{indentfirst}

\def\R{\mathbb{R}}
\def\N{\mathbb{N}}

\def\HH{\mathrm{H}}
\def\LL{\mathrm{L}}

\newcommand{\fonction}[5]{\begin{array}[t]{lrcl}#1 :&#2 &\longrightarrow &#3\\&#4& \longmapsto &#5 \end{array}}

\newcommand{\dual}[2]{\left\langle #1 , #2 \right\rangle}

\newlist{primenumerate}{enumerate}{1}
\setlist[primenumerate,1]{label={\roman*$'$}}

\title{Shape optimization for variational inequalities: the scalar Tresca friction problem}

\author{Samir Adly\footnote{Institut de recherche XLIM. UMR CNRS 7252. Universit\'e de Limoges, France. \texttt{samir.adly@unilim.fr}}, Lo\"ic Bourdin\footnote{Institut de recherche XLIM. UMR CNRS 7252. Universit\'e de Limoges, France. \texttt{loic.bourdin@unilim.fr}}, 
Fabien Caubet\footnote{Universit\'e de Pau et des Pays de l'Adour, E2S UPPA, CNRS, LMAP, UMR 5142, 64000 Pau, France. \texttt{fabien.caubet@univ-pau.fr}}, Aymeric Jacob de Cordemoy\footnote{Universit\'e de Pau et des Pays de l'Adour, E2S UPPA, CNRS, LMAP, UMR 5142, 64000 Pau, France. \texttt{aymeric.jacob-de-cordemoy@univ-pau.fr}}
}

\begin{document}

\maketitle

\begin{abstract}
This paper investigates, without any regularization or penalization procedure, a shape optimization problem involving a simplified friction phenomena modeled by a scalar Tresca friction law. Precisely, using tools from convex and variational analysis such as proximal operators and the notion of twice epi-differentiability, we prove that the solution to a scalar Tresca friction problem admits a directional derivative with respect to the shape which moreover coincides with the solution to a boundary value problem involving Signorini-type unilateral conditions. Then we explicitly characterize the shape gradient of the corresponding energy functional and we exhibit a descent direction. Finally numerical simulations are performed to solve the corresponding energy minimization problem under a volume constraint which shows the applicability of our method and our theoretical results.
\end{abstract}


\medskip



\section{Introduction}

\paragraph{Motivation}
On the one hand, shape optimization is the mathematical field whose aim is to find the optimal shape of a given object with respect to a given criterion (see, e.g.,~\cite{ALL,HENROT,SOKOZOL}). 
It is increasingly taken into account in industry in order to identify the optimal shape of a product who must satisfy some constraints. On the other hand, mechanical contact models are used to study the contact of deformable solids that touch each other on parts of their boundaries  (see, e.g.,~\cite{DAUTLIONS,KUSS,LIONS2}). Usually the contact prevents  penetration between the two rigid bodies, and possibly allows sliding modes which causes friction phenomena. A non-permeable contact can be described by the so-called {\it Signorini unilateral conditions} (see, e.g.,~\cite{15SIG,16SIG}) that take the form of inequality conditions on the contact surface, while a friction phenomenon can be described by the so-called {\it Tresca friction law} (see, e.g.,~\cite{KUSS}) which appears as a boundary condition involving nonsmooth inequalities depending on a friction threshold.

Shape optimization problems involving mechanical contact models have already been investigated in the literature (see, e.g.,~\cite{BEREM,FULM,HASKLAR,HASLINGER,HEINEMANN,HINTERMULLERLAURAIN} and references therein), and they are increasingly taken into account in industrial issues and engineering applications. Due to the involved inequalities and nonsmooth terms, the standard methods found in the literature usually consist in regularization (see, e.g.,~\cite{MAUALLJOU,CHAUDET2,LUFT}), penalization (see, e.g.,~\cite{CHAUDET}) or dualization (see~\cite[Chapter 4]{SOKOZOL} and~\cite{SOKOZOLE2}) procedures. In simple terms, regularization (resp. penalization) procedures consist in using Moreau's envelopes (resp. penalty functionals) to approximate the optimization problem associated with the model. However, both of these methods do not take into account the exact characterization of the solution and may perturb the original nature of the model. The dualization method used in~\cite{SOKOZOLE2} consists in describing the primal/dual pair as a saddle point of the associated Lagrangian. Then the dual problem leads to a characterization that involves only projection operators and thus Mignot’s theorem (see~\cite{MIGNOT}) about conical differentiability can be applied. However this method results in material/shape derivative characterizations that are implicit, as they involve dual elements. In this paper our aim is to propose a new methodology which allows to preserve the original nature of the problem, that is, without using any regularization or penalization procedure, and moreover to work only with the primal problem. Precisely our strategy is based on the theory of variational inequalities and on tools from convex and variational analysis such as the notion of proximal operator introduced by J.J.\ Moreau in~1965 (see~\cite{MOR}) and the notion of twice epi-differentiability introduced by R.T.\ Rockafellar in~1985 (see~\cite{Rockafellar}). To the best of our knowledge, this is the ﬁrst time that these concepts are applied in the context of shape optimization problems involving nonsmoothness, which makes this contribution new and original in the literature.

As a first step towards more realistic and more complex mechanical contact models, note that the present paper focuses only on a shape optimization problem involving a simplified friction phenomena modeled by a scalar Tresca friction law. The extension of our methodology to the vectorial elasticity model, or to other variational inequalities (such as Signorini-type models), will be the subject of future research.

\medskip

\paragraph{Description of the shape optimization problem and methodology}
In this paragraph, we use standard notations which are recalled in Section~\ref{rappelgeneral}. Let~$d\in \N^*$ be a positive integer which represents the dimension, and let~$f\in\HH^{1}(\R^{d})$ and~$g\in\HH^{2}(\R^{d})$ be such that $g>0$ almost everywhere (\textit{a.e.}) on~$\R^{d}$. In this paper, we consider the shape optimization problem given by
\begin{equation}\label{shapeOptim}
    \minn\limits_{ \substack{ \Omega\in \mathcal{U} \\ \vert \Omega \vert = \lambda } } \; \mathcal{J}(\Omega),
\end{equation}
where
\begin{equation*}
        \mathcal{U} :=\{ \Omega\subset\R^{d} \mid \Omega \text{ nonempty connected bounded open subset of } \R^{d}  \text{ with Lipschitz boundary} \},
\end{equation*}
with the volume constraint $\vert \Omega \vert = \lambda > 0$, where $\mathcal{J} : \mathcal{U} \to \R$ is the \textit{Tresca energy functional} defined by
\begin{equation*}
    \mathcal{J}(\Omega) := \frac{1}{2}\int_{\Omega}\left( \left\| \nabla{u_\Omega} \right\|^{2}+|u_{\Omega}|^{2}\right)+\int_{\Gamma}g|u_{\Omega}|-\int_{\Omega}fu_{\Omega},
\end{equation*}
where $\Gamma:=\partial{\Omega}$ is the boundary of $\Omega$ and where $u_\Omega \in\HH^{1}(\Omega)$ stands for the unique solution to the scalar Tresca friction problem given by
\begin{equation}\label{Trescaproblem2222}\tag{TP\ensuremath{{}_\Omega}}
\arraycolsep=2pt
\left\{
\begin{array}{rcll}
-\Delta u+u &=&f    & \text{ in } \Omega, \\
|\partial_{\nn}u|\leq g \text{ and }  u\partial_{\nn}u+g|u|&=&0  & \text{ on } \Gamma,
\end{array}
\right.
\end{equation}
for all~$\Omega \in \mathcal{U}$. Recall that, in contact mechanics, $f$ models volume forces and that the boundary condition in~\eqref{Trescaproblem2222} is known as the scalar version of the Tresca friction law (see, e.g.,~\cite[Section~1.3 Chapter~1]{GLOW2}) where $g$ is a given friction threshold. In this paper, we refer to it as the scalar Tresca friction law. Note that we focus here on minimizing the energy functional (as in~\cite{FULM,HPM, VELI}) which corresponds to maximize the compliance (see~\cite{ALL}). In simple terms, our research focuses on finding the "laziest shape" that can resist external forces, while taking into account the effect of friction on its surface.

Also recall that, for any~$\Omega \in \mathcal{U}$, the unique solution {to}~\eqref{Trescaproblem2222} is characterized by~$u_\Omega=\mathrm{prox}_{\phi_\Omega}(F_\Omega)$, where $F_\Omega \in\HH^{1}(\Omega)$ is the unique solution to the classical Neumann problem
$$
\arraycolsep=2pt
\left\{
\begin{array}{rcll}
-\Delta F+F &=&f    & \text{ in } \Omega, \\
\partial_{\nn}F&=&0 & \text{ on } \Gamma,
\end{array}
\right.
$$
and where $\mathrm{prox}_{\phi_{\Omega}} : \HH^1(\Omega) \to \HH^1(\Omega)$ stands for the proximal operator associated with the \textit{Tresca friction functional} $\phi_{\Omega} : \HH^1(\Omega) \to \R$ defined by
$$
\fonction{\phi_\Omega}{\HH^{1}(\Omega)}{\R}{v}{\displaystyle \phi_\Omega(v):=\int_{\Gamma}g|v|.}
$$We refer for instance to~\cite{4ABC} for details on existence/uniqueness and characterization of the solution to Problem~\eqref{Trescaproblem2222}.

To deal with the numerical treatment of the above shape optimization problem, a suitable expression of the shape gradient of~$\mathcal{J}$ is required. To this aim we follow the classical strategy developed in the shape optimization literature (see, e.g.,~\cite{ALL,HENROT}). Consider $\Omega_{0} \in \mathcal{U}$ and a direction~$\boldsymbol{V}\in\mathcal{C}^{1,\infty}(\R^{d},\R^{d}):=~\mathcal{C}^{1}(\R^{d},\R^{d})\cap\mathrm{W}^{1,\infty}(\R^{d},\R^{d})$. Then, for any $t\geq0$ sufficiently small such that~$\boldsymbol{\mathrm{id}}+t\boldsymbol{V}$ is a~$\mathcal{C}^{1}$-diffeomorphism of $\R^{d}$, we denote by~$\Omega_{t}:=(\boldsymbol{\mathrm{id}}+t\boldsymbol{V})(\Omega_{0}) \in \mathcal{U}$ and by~$u_{t} := u_{\Omega_t} \in\HH^{1}(\Omega_{t})$, where $\boldsymbol{\mathrm{id}} :  \R^{d}\rightarrow \R^{d}$ stands for the identity operator. To get an expression of the shape gradient of~$\mathcal{J}$ at~$\Omega_0$ in the direction~$\boldsymbol{V}$, the first step naturally consists in obtaining an expression of the derivative of the map~$t \in \R_+ \mapsto u_{t} \in\HH^{1}(\Omega_{t})$ at~$t=0$. However this map is not well defined since the codomain~$\HH^{1}(\Omega_{t})$ depends on the variable~$t$. To overcome the issue that~$u_t$ is defined on the moving domain~$\Omega_t$, we consider the change of variables~$\boldsymbol{\mathrm{id}}+t\boldsymbol{V}$ and we prove that~$\overline{u}_t:=u_{t}\circ(\boldsymbol{\mathrm{id}}+t\boldsymbol{V})\in\HH^{1}(\Omega_{0})$ is the unique solution to the perturbed scalar Tresca friction problem given by
\begin{equation*}
\arraycolsep=2pt
\left\{
\begin{array}{rcll}
-\mathrm{div}\left(\mathrm{A}_{t}\nabla\overline{u}_t\right)+\overline{u}_t\mathrm{J}_{t} &=&f_{t}\mathrm{J}_{t}    & \text{ in } \Omega_{0} , \\
\left|\mathrm{A}_{t}\nabla\overline{u}_t\cdot\boldsymbol{\nn}\right|\leq g_{t}\mathrm {J}_{\mathrm{T}_{t}}\text{ and }  \overline{u}_t\mathrm{A}_{t}\nabla\overline{u}_t\cdot\boldsymbol{\nn}+g_{t}\mathrm {J}_{\mathrm{T}_{t}}\left|\overline{u}_t\right|&=&0  & \text{ on } \Gamma_{0},
\end{array}
\right.
\end{equation*}
considered on the fixed domain~$\Omega_0$, where~$\Gamma_0 := \partial \Omega_0$, $f_{t}:=f\circ(\boldsymbol{\mathrm{id}}+t\boldsymbol{V})\in\HH^{1}(\R^{d})$, $g_{t}:=g\circ(\boldsymbol{\mathrm{id}}+t\boldsymbol{V})\in \HH^{1}(\R^{d})$ and where $\mathrm{J}_{t}$, $\mathrm{A}_{t}$ and $\mathrm {J}_{\mathrm{T}_{t}}$ are standard Jacobian terms resulting from the change of variables used in the weak variational formulation of Problem (TP${}_{\Omega_t}$) (see details in Subsection~\ref{PerturbTrescatyoefric}). Hence, the shape perturbation is shifted, via the change of variables, to the data of the scalar Tresca friction problem. 

Now, to obtain an expression of the derivative of the map~$t \in \R_+ \mapsto \overline{u}_{t} \in \HH^{1}(\Omega_0)$ at~$t=~0$, which will be denoted by~$\overline{u}'_0 \in \HH^{1}(\Omega_0)$ and called \textit{material directional derivative} (the terminology \textit{directional} has been added with respect to the literature since, in the present nonsmooth framework, the expression of $\overline{u}'_0 $ will not be linear with respect to the direction $\boldsymbol{V}$, see~Remark~\ref{remnonlinear} for details), we write that $\overline{u}_t=\mathrm{prox}_{\phi_{t}}(F_{t})
$, where $F_{t}\in\HH^{1}(\Omega_{0})$ is the unique solution to the perturbed Neumann problem
\begin{equation*}
\arraycolsep=2pt
\left\{
\begin{array}{rcll}
-\mathrm{div}\left(\mathrm{A}_{t}\nabla F_{t}\right)+F_{t}\mathrm{J}_{t} &=&f_{t}\mathrm{J}_{t}    & \text{ in } \Omega_{0} , \\
\mathrm{A}_{t}\nabla F_{t}\cdot \boldsymbol{\nn}&=&0  & \text{ on } \Gamma_{0},
\end{array}
\right.
\end{equation*}
and where $\phi_{t}$ is the perturbed Tresca friction functional given by
$$
\displaystyle\fonction{\phi_{t}}{\HH^{1}(\Omega_{0})}{\R}{v}{\displaystyle \phi_{t}(v):=\int_{\Gamma_{0}}g_{t}\mathrm {J}_{\mathrm{T}_{t}}|v|,}
$$
considered on the perturbed Hilbert space $(\HH^{1}(\Omega_{0}),\dual{\cdot}{\cdot}_{\mathrm{A}_{t},\mathrm{J}_{t}})$ (see details on the perturbed scalar product in Subsection~\ref{boundary}). To deal with the differentiability (in a generalized sense) of the parameterized proximal operator~$\mathrm{prox}_{\phi_{t}} : \HH^{1}(\Omega_{0}) \to \HH^{1}(\Omega_{0})$ we invoke the notion of \textit{twice epi-differentiability} for convex functions introduced by R.T.~Rockafellar in~1985 (see~\cite{Rockafellar}) which leads to the \textit{protodifferentiability} of the corresponding proximal operators. Actually, since the work by R.T.~Rockafellar deals only with non-parameterized convex functions, we will use instead the recent work~\cite{8AB} where the notion of twice epi-differentiability has been adapted to parameterized convex functions.

Before listing the main theoretical results obtained in the present paper thanks to the above strategy, let us mention that the sensitivity analysis of the scalar Tresca friction problem~\eqref{Trescaproblem2222} with respect to perturbations of~$f$ and~$g$ has already been performed in our previous paper~\cite{BCJDC}. However, since it was done in a general context (not in the specific context of shape optimization), the previous paper~\cite{BCJDC} considered only the case where $\mathrm{J}_{t}=\mathrm{J}_{\mathrm{T}_{t}}=1$ and $\mathrm{A_{t}}=\mathrm{I}$ is the identity matrix of~$\R^{d \times d}$ and thus the scalar product $\dual{\cdot}{\cdot}_{\mathrm{A}_{t},\mathrm{J}_{t}}$ was independent of the parameter~$t$. Hence some nontrivial adjustments are required to deal with the $t$-dependent context of the present work. We refer to Subsection~\ref{PerturbTrescatyoefric} for details.

Finally, notice that, in this paper, we do not prove theoretically the existence of a solution to the shape optimization problem~\eqref{shapeOptim}. The interested reader can find some related existence results (for very specific geometries in the two dimensional case) in~\cite{HASKLAR}.

\medskip

\paragraph{Main theoretical results}
Our main theoretical results, stated in Theorems~\ref{Lagderiv} and~\ref{shapederivofJ}, are summarized below. However, to make their expressions more explicit and elegant, we present them under certain additional regularity assumptions, such as~$u_0 \in \HH^3(\Omega_0)$, within the framework of Corollaries~\ref{LagderivSigno},~\ref{shapederivaticesigno} and~\ref{shapederivofJ3}, making them more suitable for this introduction.
\begin{enumerate}
  \item[(i)] Under some appropriate assumptions described in Corollary~\ref{LagderivSigno}, the material directional derivative $\overline{u}'_0\in\HH^{1}(\Omega_{0})$ is the unique weak solution to the scalar Signorini problem given by
\begin{equation*}
\arraycolsep=2pt
\left\{
\begin{array}{rcll}
-\Delta \overline{u}'_0+\overline{u}'_0 & = & -\Delta\left(\boldsymbol{V}\cdot\nabla{u_{0}}\right)+\boldsymbol{V}\cdot\nabla{u_{0}}  & \text{ in } \Omega_{0} , \\
\overline{u}'_0 & = & 0  & \text{ on } \Gamma^{u_{0},g}_{\mathrm{D}}, \\
\partial_{\nn} \overline{u}'_0 & = & h^m(\boldsymbol{V})  & \text{ on } \Gamma^{u_{0},g}_{\mathrm{N}}, \\
 \overline{u}'_0\leq0\text{, } \partial_{\nn}\overline{u}'_0\leq h^m(\boldsymbol{V}) \text{ and } \overline{u}'_0\left(\partial_{\nn}\overline{u}'_0-h^m(\boldsymbol{V})\right) & = & 0  & \text{ on } \Gamma^{u_{0},g}_{\mathrm{S-}}, \\
 \overline{u}'_0\geq0\text{, } \partial_{\nn}\overline{u}'_0\geq h^m(\boldsymbol{V}) \text{ and } \overline{u}'_0\left(\partial_{\nn}\overline{u}'_0- h^m(\boldsymbol{V})\right) & = & 0  & \text{ on } \Gamma^{u_{0},g}_{\mathrm{S+}},
\end{array}
\right.
\end{equation*}
where $h^m(\boldsymbol{V}):= (\frac{\nabla{g}}{g}\cdot \boldsymbol{V}-\nabla{\boldsymbol{V}}\boldsymbol{\nn}\cdot\boldsymbol{\nn})\partial_{\nn}u_{0}+ (\nabla{\boldsymbol{V}}+\nabla{\boldsymbol{V}}^{\top})\nabla{u_{0}}\cdot\boldsymbol{\nn}\in\LL^{2}(\Gamma_{0})$, where~$\nabla \boldsymbol{V}$ stands for the standard Jacobian matrix of~$\boldsymbol{V}$, and where $\Gamma_0$ is decomposed (up to a null set) as~$\Gamma^{u_{0},g}_{\mathrm{N}}\cup
\Gamma^{u_{0},g}_{\mathrm{D}}\cup\Gamma^{u_{0},g}_{\mathrm{S-}}\cup\Gamma^{u_{0},g}_{\mathrm{S+}}$ (see details in Theorem~\ref{Lagderiv}). Recall that the boundary conditions on~$\Gamma^{u_{0},g}_{\mathrm{S-}}$ and~$\Gamma^{u_{0},g}_{\mathrm{S+}}$ are known as the scalar versions of the Signorini unilateral conditions (see, e.g.,~\cite[Section~1]{LIONS2}).
\item[(ii)]  We deduce in Corollary~\ref{shapederivaticesigno} that, under appropriate assumptions, \textit{the shape directional derivative}, defined by~$u'_{0}:=\overline{u}'_0-\nabla{u_{0}}\cdot\boldsymbol{V}\in\HH^{1}(\Omega_{0})$ (which roughly corresponds to the derivative of the map $t\in\R_{+} \mapsto u_{t} \in \HH^{1}(\Omega_t)$ at~$t=0$), is the unique weak solution to the scalar Signorini problem given by
$$
\hspace{-0.5cm}
\arraycolsep=2pt
\left\{
\begin{array}{rcll}
-\Delta u_{0}'+u_{0}' & = & 0  & \text{ in } \Omega_{0} , \\
u_{0}' & = & -\boldsymbol{V}\!\cdot\!\nabla{u_{0}} & \text{ on } \Gamma^{u_{0},g}_{\mathrm{D}}, \\
\partial_{\nn} u_{0}' & = & h^s(\boldsymbol{V})  & \text{ on } \Gamma^{u_{0},g}_{\mathrm{N}}, \\
\! u_{0}'\leq-\boldsymbol{V}\!\cdot\!\nabla{u_{0}}\text{, } \partial_{\nn}u_{0}'\leq h^s(\boldsymbol{V}) \text{ and } \left(u_{0}'+\boldsymbol{V}\!\cdot\!\nabla{u_{0}}\right)\left(\partial_{\nn}u_{0}'-h^s(\boldsymbol{V})\right) & = & 0  & \text{ on } \Gamma^{u_{0},g}_{\mathrm{S-}}, \\
\! u_{0}'\geq-\boldsymbol{V}\!\cdot\!\nabla{u_{0}}\text{, } \partial_{\nn}u_{0}'\geq h^s(\boldsymbol{V}) \text{ and } \left(u_{0}'+\boldsymbol{V}\!\cdot\!\nabla{u_{0}}\right)\left(\partial_{\nn}u_{0}'- h^s(\boldsymbol{V})\right) & = & 0  & \text{ on } \Gamma^{u_{0},g}_{\mathrm{S+}},
\end{array}
\right.
$$
where $h^s(\boldsymbol{V}):=\boldsymbol{V}\cdot \boldsymbol{\nn} (\partial_{\nn}  (\partial_{\nn}u_0 )-\frac{\partial^2 u_0}{\partial \mathrm{n}^2} )+\nabla_{\Gamma_0} u_{0} \cdot\nabla_{\Gamma_0}(\boldsymbol{V}\cdot\boldsymbol{\mathrm{\nn}})-g\nabla{(\frac{\partial_{\nn}u_{0}}{g})}\cdot\boldsymbol{V}\in\LL^{2}(\Gamma_{0})$.
\item[(iii)] Finally the two previous items are used to obtain Corollary~\ref{shapederivofJ3} asserting that, under appropriate assumptions, the shape gradient of $\mathcal{J}$ at~$\Omega_{0}$ in the direction~$\boldsymbol{V}$ is given by
$$ \small{ \mathcal{J}'(\Omega_{0})(\boldsymbol{V})= \int_{\Gamma_{0}}\boldsymbol{V}\cdot\boldsymbol{\nn}\left(\frac{\left\|\nabla{u_{0}}\right\|^{2}+\left|u_{0}\right|^{2}}{2}-fu_{0}+Hg\left| u_{0} \right| -\partial_{\nn}\left(u_{0}\partial_{\nn}u_{0}\right)+gu_{0}\nabla{\left(\frac{\partial_{\nn}u_{0}}{g}\right)}\cdot\boldsymbol{\nn}\right), }
$$
where $H$ stands for the mean curvature of $\Gamma_{0}$. We emphasize that, with the Tresca energy functional~$\mathcal{J}$ considered in the present work, we obtain that $\mathcal{J}'(\Omega_0)$ depends only on $u_0$ (and not on $u’_0$). As a consequence its expression is explicit (and also linear) with respect to the direction $\boldsymbol{V}$. In particular this implies that there is no need to introduce any adjoint problem to perform numerical simulations (see Remark~\ref{remarkadjoint} for details). 
\end{enumerate}
\medskip

\paragraph{Application to shape optimization and numerical simulations} 

The expression of the shape gradient of $\mathcal{J}$ stated in (iii) allows us to exhibit an explicit descent direction of $\mathcal{J}$ (see Section~\ref{numericalsim} for details). Hence, using this descent direction together with a basic Uzawa algorithm to take into account the volume constraint, we perform in Section~\ref{numericalsim} numerical simulations to solve the shape optimization problem~\eqref{shapeOptim} on a two-dimensional example. Furthermore, we present several numerical results with different values of $g$, allowing us to emphasize an interesting behavior of the optimal shape. Precisely, in our example, it seems to transit from the optimal shape when one replaces the Tresca problem and its energy functional by Dirichlet ones when $g$ goes to infinity pointwisely, to the optimal shape when one replaces the Tresca problem and its energy functional by Neumann ones when $g$ goes to zero pointwisely.

\medskip

\paragraph{Organization of the paper}
The paper is organized as follows. Section~\ref{rappelgeneral} is dedicated to some basic recalls from convex, variational and functional analysis, differential geometry and boundary value problems involved all along the paper. In Section~\ref{mainresult}, we state and prove our main theoretical results. Finally, in Section~\ref{numericalsim}, numerical simulations are performed to solve the shape optimization problem~\eqref{shapeOptim} on a two-dimensional example.\\

\section{Preliminaries}\label{rappelgeneral}

\subsection{Reminders on proximal operator and twice epi-differentiability}\label{rappelconvex}

For notions and results recalled in this subsection, we refer to standard references from convex and variational analysis literature such as~\cite{BREZ2,MINTY,ROCK2} and~\cite[Chapter~12]{ROCK}. In what follows, $(\mathcal{H}, \dual{\cdot}{\cdot}_{\mathcal{H}})$ stands for a general real Hilbert space. The \textit{domain} and the \textit{epigraph} of an extended real value function~$\psi : \mathcal{H}\rightarrow \mathbb{R}\cup\left\{\pm \infty \right\}$ are respectively defined by
$$
\mathrm{dom}\left(\psi\right):=\left\{x\in \mathcal{H} \mid \psi(x)<+\infty \right\} \quad \text{and} \quad
\mathrm{epi}\left(\psi\right):=\left\{(x,t)\in \mathcal{H}\times\mathbb{R}\mid \psi(x)\leq t\right\}.
$$
Recall that $\psi$ is said to be \textit{proper} if $\mathrm{dom}(\psi)\neq \emptyset$ and $\psi(x)>-\infty$ 
for all~$x\in\mathcal{H}$, and that $\psi$ is convex (resp.\ lower semi-continuous) if and only if~$\mathrm{epi}(\psi)$ is a convex (resp.\ closed) subset of~$\mathcal{H}\times\R$. When $\psi$ is proper, we denote by $\partial{\psi}  :  \mathcal{H} \rightrightarrows \mathcal{H}$ its \textit{convex subdifferential operator}, defined by 
 $$
 \partial{\psi}(x):=\left\{y\in\mathcal{H} \mid \forall z\in\mathcal{H}\text{, } \dual{y}{z-x}_{\mathcal{H}}\leq \psi(z)-\psi(x)\right\},
 $$
when $x\in \mathrm{dom}(\psi)$, and by $\partial{\psi}(x):= \emptyset$ whenever $x\notin \mathrm{dom}(\psi)$. The notion of proximal operator has been introduced by J.J.\ Moreau in 1965 (see~\cite{MOR}) as follows.

\begin{myDefn}\label{proxi}
The \textit{proximal operator} associated with a proper, lower semi-continuous and convex function $\psi  : \mathcal{H} \rightarrow \R\cup\left\{+\infty\right\}$ is the map~$\mathrm{prox}_{\psi}  :  \mathcal{H} \rightarrow \mathcal{H}$ defined by
 $$
      \mathrm{prox}_{\psi}(x):=\underset{y\in\mathcal{H}}{\argmin}\left[ \psi(y)+\frac{1}{2}\left \| y-x \right \|^{2}_{\mathcal{H}}\right]=(\mathrm{id}+\partial \psi)^{-1}(x),
 $$
for all $x\in\mathcal{H}$, where $\mathrm{id}  :  \mathcal{H}\rightarrow \mathcal{H}$ stands for the identity operator.
\end{myDefn}

Recall that, if $\psi  :  \mathcal{H} \rightarrow \R\cup\left\{+\infty\right\}$ is a proper, lower semi-continuous and convex function, then its subdifferential~$\partial{\psi}$ is a maximal monotone operator (see, e.g.,~\cite{ROCK2}), and thus its proximal operator~$\mathrm{prox}_{\psi} : \mathcal{H} \rightarrow \mathcal{H}$ is well-defined, single-valued and nonexpansive, i.e. Lipschitz continuous with modulus $1$ (see, e.g.,~\cite[Chapter II]{BREZ2}).

As mentioned in Introduction, the unique solution to the scalar Tresca friction problem considered in this paper can be expressed via the proximal operator of the associated Tresca friction functional $\phi_\Omega$. Therefore the shape sensitivity analysis of this problem is related to the differentiability (in a generalized sense) of the involved proximal operator. To investigate this issue, we will use the notion of twice epi-differentiability introduced by R.T.~Rockafellar in~1985 (see~\cite{Rockafellar}) defined as the Mosco epi-convergence of second-order difference quotient functions. Our aim in what follows is to provide reminders and backgrounds on these notions for the reader's convenience. For more details, we refer to~\cite[Chapter 7, Section B p.240]{ROCK} for the finite-dimensional case and to~\cite{DO} for the infinite-dimensional case. The strong (resp.\ weak) convergence of a sequence in~$\mathcal{H}$ will be denoted by~$\rightarrow$ (resp.\ $\rightharpoonup$) and note that all limits with respect to~$t$ will be considered for~$t \to 0^+$.

\begin{myDefn}[Mosco convergence]\label{limitemuch}
The \textit{outer}, \textit{weak-outer}, \textit{inner} and \textit{weak-inner limits} of a parameterized family~$(S_{t})_{t>0}$ of subsets of $\mathcal{H}$ are respectively defined by
\begin{eqnarray*}
      \mathrm{lim}\sup S_{t}&:=&\left\{ x\in \mathcal{H} \mid \exists (t_{n})_{n\in\mathbb{N}}\rightarrow 0^{+}, \; \exists \left(x_{n}\right)_{n\in\mathbb{N}}\rightarrow x, \; \forall n\in\mathbb{N}, \; x_{n}\in S_{t_{n}}\right\},\\
     \mathrm{w}\text{-}\mathrm{lim}\sup S_{t}&:=&\left\{ x\in \mathcal{H} \mid \exists (t_{n})_{n\in\mathbb{N}}\rightarrow 0^{+}, \; \exists \left(x_{n}\right)_{n\in\mathbb{N}}\rightharpoonup x, \; \forall n\in\mathbb{N}, \; x_{n}\in S_{t_{n}}\right\},\\
     \mathrm{lim}\inf S_{t}&:=&\left\{ x\in \mathcal{H} \mid \forall (t_{n})_{n\in\mathbb{N}}\rightarrow 0^{+}, \; \exists \left(x_{n}\right)_{n\in\mathbb{N}}\rightarrow x, \; \exists N\in\mathbb{N}, \; \forall n\geq N, \; x_{n}\in S_{t_{n}}\right\},\\
     \mathrm{w}\text{-}\mathrm{lim}\inf S_{t}&:=&\left\{ x\in \mathcal{H} \mid \forall (t_{n})_{n\in\mathbb{N}}\rightarrow 0^{+}, \; \exists \left(x_{n}\right)_{n\in\mathbb{N}}\rightharpoonup x, \; \exists N\in\mathbb{N}, \; \forall n\geq N, \; x_{n}\in S_{t_{n}}\right\}.
\end{eqnarray*}
The family~$(S_{t})_{t>0}$ is said to be \textit{Mosco convergent} if~$
\mathrm{w}\text{-}\mathrm{lim}\sup S_{t}\subset\mathrm{lim}\inf S_{t}
$. In that case all the previous limits are equal and we write
$$
     \mathrm{M}\text{-}\mathrm{lim}~S_{t}:=\mathrm{lim}\inf S_{t}=\mathrm{lim}\sup S_{t}=\mathrm{w}\text{-}\mathrm{lim}\inf S_{t}=\mathrm{w}\text{-}\mathrm{lim}\sup S_{t}.
$$
\end{myDefn}

\begin{myDefn}[Mosco epi-convergence]\label{Mosco}
  Let $(\psi_{t})_{t>0}$ be a parameterized family of functions~$\psi_{t}  : \mathcal{H}\rightarrow \mathbb{R}\cup\left\{\pm \infty \right\}$ for all $t>0$.
 We say that $(\psi_{t})_{t>0}$ is \textit{Mosco epi-convergent} if~$(\mathrm{epi}(\psi_{t}))_{t>0}$ is Mosco convergent in~$\mathcal{H}\times \R$. Then we denote by $\mathrm{ME}\text{-}\mathrm{lim}~ \psi_{t}  :  \mathcal{H}\rightarrow \mathbb{R}\cup\left\{\pm \infty \right\}$ the function characterized by its epigraph~$\mathrm{epi}\left(\mathrm{ME}\text{-}\mathrm{lim}~\psi_{t}\right):=\mathrm{M}\text{-}\mathrm{lim}$ $\displaystyle \mathrm{epi}\left(\psi_{t}\right)$ and we say that $(\psi_{t})_{t>0}$ Mosco epi-converges to~$\mathrm{ME}\text{-}\mathrm{lim}~\psi_{t}$.
 \end{myDefn}

\begin{myRem}\normalfont
In Definition \ref{Mosco}, the abbreviation ME stands for the \textit{Mosco Epi-convergence} (which is related to functions), while the abbreviation M stands for the \textit{Mosco convergence} (related to subsets).
\end{myRem}

The notion of twice epi-differentiability was originally introduced for nonparameterized convex functions. However, as mentioned in Introduction, the framework of the present paper requires an extended version to parameterized convex functions which has recently been developed in~\cite{8AB}. To provide recalls on this extended notion, when considering a function~$\Psi  :  \mathbb{R}_{+}\times \mathcal{H}\rightarrow \mathbb{R}\cup\left\{+\infty\right\}$ such that, for all $t\geq0$, $\Psi(t,\cdot) : \mathcal{H}\rightarrow \mathbb{R} \cup\left\{+\infty\right\}$ is a proper function, we will make use of the following two notations: $\partial \Psi(0,\mathord{\cdot} )(x)$ stands for the convex subdifferential operator at~$x\in\mathcal{H}$ of the function~$\Psi(0,\cdot)$, and for each $t\geq 0$, $\Psi^{-1}(t , \mathbb{R}):=\left\{ x\in\mathcal{H}\mid \; \Psi(t,x)\in\R \right\}$ and $\Psi^{-1}(\mathord{\cdot} , \mathbb{R}):=\displaystyle\cap_{t\geq 0}\Psi^{-1}(t , \mathbb{R})  $.

 \begin{myDefn}[Twice epi-differentiability depending on a parameter]\label{epidiffpara}
Let~$\Psi  :  \mathbb{R}_{+}\times \mathcal{H}\rightarrow \mathbb{R} \cup\left\{+\infty\right\}$ be a function such that, for all $t\geq0$, $\Psi(t,\cdot) : \mathcal{H}\rightarrow \mathbb{R} \cup\left\{+\infty\right\}$ is a proper lower semi-continuous convex function. Then $\Psi$ is said to be \textit{twice epi-differentiable} at $x\in \Psi^{-1}(\mathord{\cdot} , \mathbb{R})$ for~$y\in\partial \Psi(0,\mathord{\cdot} )(x)$ if the family of second-order difference quotient functions $(\Delta_{t}^{2}\Psi(x|y))_{t>0}$ defined by
$$
  \fonction{\Delta_{t}^{2}\Psi(x|y) }{\mathcal{H}}{\mathbb{R}\cup\left\{+\infty\right\}}{z}{ \Delta_{t}^{2}\Psi(x|y) (z) := \displaystyle\frac{\Psi(t,x+t z)-\Psi(t,x)-t\dual{ y}{z}_{\mathcal{H}}}{t^{2}},}
$$
for all $t>0$, is Mosco epi-convergent. In that case we denote by
$$
\mathrm{D}_{e}^{2}\Psi(x|y):=\mathrm{ME}\text{-}\mathrm{lim}~\Delta_{t}^{2}\Psi(x|y) ,
$$
which is called the \textit{second-order epi-derivative} of $\Psi$ at $x$ for $y$.
\end{myDefn}

\begin{myRem}\normalfont
If the real-valued function $\Psi$ is $t$-independent, Definition~\ref{epidiffpara} recovers
the classical notion of twice epi-differentiability originally introduced in~\cite{Rockafellar} (up to
the multiplicative constant~$\frac{1}{2}$).  
\end{myRem}

\begin{myRem}\label{lscC}\normalfont 
It is well-known that the convexity and the lower-semicontinuity are preserved by the Mosco epi-convergence. However, the properness of the Mosco epi-limit may fail even if the sequence is proper. If, for each $t\geq0$, $\Psi(t,\cdot) : \mathcal{H}\rightarrow \mathbb{R} \cup\left\{+\infty\right\}$ is a proper, lower semi-continuous and convex function, then the Mosco epi-limi $\mathrm{D}_{e}^{2}\Psi(x|y)$ (when it exists) is also lower semi-continuous and convex function. However, it may be possible that there exists some $z\in \mathcal{H}$ such that $\mathrm{D}_{e}^{2}\Psi(x|y)(z)=-\infty$ (see, e.g.,~\cite[Example 4.4 p.1711]{8AB}).
\end{myRem}

To illustrate the notion of twice epi-differentiability, two examples extracted from~\cite[Lemma~5.2 p.1717]{8AB} are given below. The first example is about a $t$-independent function which will be useful in this paper (see Lemma~\ref{épidiffgabs}) and the second one concerns a $t$-dependent function.

\begin{myExa}\label{epidiffabs}\normalfont
The classical absolute value map $\left|\cdot\right|  :  \mathbb{R}\rightarrow\mathbb{R}$, which is a proper lower semi-continuous convex function on $\mathbb{R}$, is twice epi-differentiable at any $x\in\R$ for any~$y\in\partial{\left|\cdot\right|}(x)$, and its second-order epi-derivative is given by~$ \mathrm{D}_{e}^{2}|\mathord{\cdot} |(x|y)=\iota_{\mathrm{K}_{x,y}}$, where $\mathrm{K}_{x,y}$ is the nonempty closed convex subset of $\R$ defined by
$$
\mathrm{K}_{x,y}:=
\left\{
\begin{array}{lcll}
\R	&   & \text{ if } x\neq0 , \\
\mathbb{R}^{-}	&   & \text{ if } x=0 \text{ and } y=-1 , \\
\mathbb{R}^{+}	&   & \text{ if } x=0 \text{ and } y=1 , \\
\left\{0\right\}	&   & \text{ if } x=0 \text{ and } y\in (-1,1),
\end{array}
\right.
$$
and where~$\iota_{\mathrm{K}_{x,y}}$ stands for the indicator function of~$\mathrm{K}_{x,y}$, defined by $\iota_{\mathrm{K}_{x,y}}(z):=0$ if $z\in\mathrm{K}_{x,y}$, and~$\iota_{\mathrm{K}_{x,y}}(z):=+\infty$ otherwise.
\end{myExa}

\begin{myExa}\normalfont
Consider the function~$\Psi : \R_+ \times \R \to \R$ defined by~$\Psi(t,x) := \vert x-t^2 \vert$ for all~$(t,x) \in \R_+ \times \R$. For each $t\geq0$, $\Psi(t,\cdot)$ is a proper, lower semi-continuous and convex function. For all~$x\in\R$ and all $y\in\partial\Psi(0,\cdot)(x)$, $\Psi$ is twice epi-differentiable at $x$ for $y$ and its second-order epi-derivative is given by 
$$
\mathrm{D}_{e}^{2}\Psi(x|y)=\left\{
\begin{array}{lcll}
\iota_{\R}&     & \text{ if } x\neq0, \\
\iota_{\R^{-}}	&   & \text{ if } x=0 \text{ and } y=-1, \\
\iota_{\R^{+}}-2	&   & \text{ if } x=0 \text{ and } y=1, \\
\iota_{\left\{0\right\}}-y-1	&   & \text{ if } x=0 \text{ and } y\in (-1,1).
\end{array}
\right.
$$
\end{myExa}

Finally the next proposition (which can be found in~\cite[Theorem 4.15 p.1714]{8AB}) is the key point to derive our main results in the present work.

\begin{myProp}\label{TheoABC2018}
Let~$\Psi  :  \mathbb{R}_{+}\times \mathcal{H}\rightarrow \mathbb{R} \cup \left\{+\infty\right\}$ be a function such that, for all $t\geq0$, $\Psi(t,\cdot) : \mathcal{H}\rightarrow \mathbb{R} \cup\left\{+\infty\right\}$ is a proper, lower semi-continuous and convex function. Let $F  :  \mathbb{R}_{+}\rightarrow \mathcal{H}$ and~$u  :  \mathbb{R}_{+}\rightarrow \mathcal{H}$ be defined by
$$
    u(t):=\mathrm{prox}_{\Psi(t,\mathord{\cdot} )}(F(t)),
$$
for all~$t\geq 0$. If the conditions 
\begin{enumerate}
    \item[{\rm (i)}]  $F$ is differentiable at $t=0$;
    \item[{\rm (ii)}]  $\Psi$ is twice epi-differentiable at $u(0)$ for $F(0)-u(0)\in\partial \Psi(0,\mathord{\cdot} )(u(0))$;
    \item[{\rm (iii)}]  $\mathrm{D}_{e}^{2}\Psi(u(0)|F(0)-u(0))$ is a proper function on $\mathcal{H}$;
\end{enumerate}
are satisfied, then $u$ is differentiable at $t=0$ with
$$
u'(0)=\mathrm{prox}_{\mathrm{D}_{e}^{2}\Psi(u(0)|F(0)-u(0))}(F'(0)).
$$
\end{myProp}

\subsection{Reminders on differential geometry}
Let $d\in\mathbb{N}^{*}$ be a positive integer, $\Omega$ be a nonempty bounded connected open subset of~$\R^{d}$ with a Lipschitz boundary $\Gamma :=\partial{\Omega}$ and $\boldsymbol{\mathrm{\nn}}$ be the outward-pointing unit normal vector to $\Gamma$. In the whole paper we denote by $\mathcal{C}^{\infty}_0(\Omega)$ the set of functions that are infinitely differentiable with
compact support in $\Omega$, by $\mathcal{C}^{\infty}_0(\Omega)'$ the set of distributions on $\Omega$, for $(m,p) \in \mathbb{N}\times\mathbb{N}^*$, by~$\mathrm{W}^{m,p}(\Omega)$, $\LL^{2}(\Gamma)$, $\HH^{1/2}(\Gamma)$, $\HH^{-1/2}(\Gamma)$, the usual Lebesgue and Sobolev spaces endowed with their standard norms, and we denote by $\HH^{m}(\Omega):=\mathrm{W}^{m,2}(\Omega)$ and by~$\HH_{\mathrm{div}}(\Omega):= \{ \boldsymbol{w}\in ( \mathrm{L}^{2}(\Omega) )^{d} \mid \mathrm{div} (\boldsymbol{w})\in \mathrm{L}^{2}(\Omega) \}$. The next proposition, known as \textit{divergence formula}, can be found in~\cite[Theorem 4.4.7 p.104]{ALLNUM}.

\begin{myProp}[Divergence formula]\label{divergenceformula}
If $\boldsymbol{w}\in \HH_{\mathrm{div}}(\Omega)$,
then $\boldsymbol{w}$ admits a normal trace, denoted by~$\boldsymbol{w}\cdot \boldsymbol{\mathrm{\nn}} \in \HH^{-1/2}(\Gamma)$, satisfying
$$
\displaystyle\int_{\Omega}\mathrm{div}(\boldsymbol{w})v+\int_{\Omega}\boldsymbol{w}\cdot\nabla v=\dual{\boldsymbol{w}\cdot \boldsymbol{\mathrm{\nn}}}{v}_{\HH^{-1/2}(\Gamma)\times \HH^{1/2}(\Gamma)}, \qquad\forall v \in \HH^{1}(\Omega).
$$
\end{myProp}

The following propositions will be useful and their proofs can be found in~\cite{HENROT}.

\begin{myProp}\label{divergencelaplacien}
Let $\boldsymbol{V}\in\mathcal{C}^{1}(\R^{d},\R^{d}) \cap \mathrm{W}^{1,\infty}(\R^{d},\R^{d})$ and $v\in\HH^{1}(\Omega)$ such that $\Delta v\in\LL^{2}(\Omega)$. Then the equality
$$
\Delta\left(\boldsymbol{V}\cdot\nabla{v}\right)=\mathrm{div}\left(\left(\Delta v\right)\boldsymbol{V}-\mathrm{div}(\boldsymbol{V})\nabla{v}+(\nabla{\boldsymbol{V}}+\nabla{\boldsymbol{V}}^{\top})\nabla{v} \right),
$$
holds true in $\mathcal{C}^{\infty}_0(\Omega)'$.
\end{myProp}

\begin{myProp}\label{intbord}
Assume that~$\Gamma$ is of class $\mathcal{C}^2$ and let $\boldsymbol{V} \in \mathcal{C}^{1}(\R^{d},\R^{d})$. It holds that
\begin{equation*}
    \int_{\Gamma}(\boldsymbol{V}\cdot\nabla{v}+v\mathrm{div}_{\Gamma}(\boldsymbol{V}))=\int_{\Gamma}\boldsymbol{V}\cdot\boldsymbol{\nn}(\partial_{\nn}v+Hv), \qquad\forall v \in \mathrm{W}^{2,1}(\Omega),
\end{equation*}
where $\mathrm{div}_{\Gamma}(\boldsymbol{V}):=\mathrm{div}(\boldsymbol{V})-(\nabla{ \boldsymbol{V}}\boldsymbol{\nn} \cdot \boldsymbol{\nn}) \in \LL^\infty(\Gamma)$ is the {tangential divergence} of $\boldsymbol{V}$, $\partial_\nn v := \nabla v \cdot \boldsymbol{\mathrm{\nn}} \in \LL^1(\Gamma)$ is the normal derivative of~$v$, and $H$ stands for the \textit{mean curvature} of $\Gamma$.
\end{myProp}

\begin{myProp}\label{beltrami}
Assume that $\Gamma$ is of class $\mathcal{C}^2$ and let $w\in\HH^{3}(\Omega)$. It holds that 
\begin{equation*}\label{prop291}
    \Delta w=\Delta_{\Gamma}w+H\partial_{\nn}w+\frac{\partial^2 w}{\partial \mathrm{n}^2} \qquad \text{\textit{a.e.}\ on } \Gamma,
\end{equation*}
where $\Delta_{\Gamma}w\in\LL^{2}(\Gamma)$ stands for the \textit{Laplace-Beltrami operator} of~$w$ (see, e.g.,~\cite[Definition~5.4.11 p.196]{HENROT}), and $\frac{\partial^2 w}{\partial \mathrm{n}^2}:=\mathrm{D}^{2}(w) \boldsymbol{\mathrm{\nn}} \cdot \boldsymbol{\mathrm{\nn}} \in \LL^2(\Gamma)$, where $\mathrm{D}^{2}(w)$ stands for the Hessian matrix of $w$. Moreover it holds that
\begin{equation*}
\int_{\Gamma}v\Delta_{\Gamma}w=-\int_{\Gamma}\nabla_{\Gamma}v\cdot\nabla_{\Gamma}w, \qquad \forall v \in \HH^{2}(\Omega),
\end{equation*}
where $\nabla_{\Gamma}v:=\nabla{v}-(\partial_{\nn}v)\boldsymbol{\nn} \in \HH^{1/2}(\Gamma,\R^{d})$ stands for the \textit{tangential gradient} of~$v$.
\end{myProp}
\subsection{Reminders on three basic nonlinear boundary value problems}\label{boundary}
As mentioned in Introduction, the major part of the present work consists in performing the sensitivity analysis of a scalar Tresca friction problem with respect to shape perturbation. To this aim three classical boundary value problems will be involved: a Neumann problem, a scalar Signorini problem and, of course, a scalar Tresca friction problem. Our aim in this subsection is to recall basic notions and results concerning these three boundary value problems for the reader's convenience. Since the proofs are very similar to the ones detailed in our paper~\cite{4ABC}, they will be omitted here.

Let $d\in\mathbb{N}^{*}$ be a positive integer and $\Omega$ be a nonempty bounded connected open subset of~$\R^{d}$ with a Lipschitz continuous boundary $\Gamma :=\partial{\Omega}$. Consider also $h\in\LL^{2}(\Omega)$, $k\in\LL^{2}(\Omega)$, $\ell\in\LL^{2}(\Gamma)$, $w\in\HH^{1/2}(\Gamma)$ and~$\mathrm{M}\in\LL^{\infty}(\Omega,\R^{d \times d})$ satisfying
$$ h\geq \alpha \text{ \textit{a.e.} on }\Omega	 \qquad \text{and} \qquad
 \mathrm{M}(x)y\cdot y\geq \gamma \|y \|^{2}, \quad \forall y \in \R^d,
 $$
for some $\alpha>0$, $\gamma>0$, where $M(x)$ is a symmetric matrix for almost every~$x \in \Omega$, and where~$\Vert \cdot \Vert$ stands for the usual Euclidean norm of~$\R^d$. From those assumptions, note that the map 
$$
\displaystyle\fonction{\dual{\cdot}{\cdot}_{\mathrm{M},h}}{\HH^{1}(\Omega)\times\HH^{1}(\Omega)}{\R}{(v_1,v_2)}{\displaystyle \dual{v_1}{v_2}_{\mathrm{M},h}:=\int_{\Omega} \mathrm{M}\nabla{v_1}\cdot\nabla{v_2}+\int_{\Omega}v_1 v_2 h,}
$$
is a scalar product on $\HH^{1}(\Omega)$. 

\subsubsection{A Neumann problem}
Consider the Neumann problem given by
\begin{equation} \label{PbDirichletArt2}\tag{NP}
\arraycolsep=2pt
\left\{
\begin{array}{rcll}
-\mathrm{div}(\mathrm{M}\nabla{F})+Fh  & = & k   & \text{ in } \Omega , \\
\mathrm{M}\nabla{F}\cdot\boldsymbol{\nn} & = & \ell  & \text{ on } \Gamma,\\
\end{array}
\right.
\end{equation}
where the data have been introduced at the beginning of Subsection~\ref{boundary}.

\begin{myDefn}[Solution to the Neumann problem]
A (strong) solution to the Neumann problem~\eqref{PbDirichletArt2} is a function $F\in\HH^{1}(\Omega)$ such that $-\mathrm{div}(\mathrm{M}\nabla{F})+Fh=k $ in $\mathcal{C}^{\infty}_0(\Omega)'$ and $\mathrm{M}\nabla{F}\cdot\boldsymbol{\nn}\in \LL^{2}(\Gamma)$ with $\mathrm{M}\nabla{F}\cdot\boldsymbol{\nn}=\ell$ \textit{a.e.} on $\Gamma$.
\end{myDefn}

\begin{myDefn}[Weak solution to the Neumann problem]\label{SolufaibleDN}
A weak solution to the Neumann problem~\eqref{PbDirichletArt2} is a function $F\in\HH^{1}(\Omega)$ such that
\begin{equation*}
    \displaystyle\int_{\Omega}\mathrm{M}\nabla{ F}\cdot\nabla{ v}+\int_{\Omega}Fvh=\int_{\Omega}kv+\int_{\Gamma}\ell v, \qquad \forall v\in\HH^{1}(\Omega).
\end{equation*}
\end{myDefn}
\begin{myProp}\label{DNequiart}
A function $F\in\HH^{1}(\Omega)$ is a (strong) solution to the Neumann problem~\eqref{PbDirichletArt2} if and only if $F$ is a weak solution to the Neumann problem~\eqref{PbDirichletArt2}.
\end{myProp}

From the assumptions on $\mathrm{M}$ and $h$ and using the Riesz representation theorem, one can easily get the following existence/uniqueness result.

\begin{myProp}\label{existenceunicitéDN}
The Neumann problem~\eqref{PbDirichletArt2} possesses a unique (strong) solution $F\in \HH^{1}(\Omega).$
\end{myProp}

\subsubsection{A scalar Signorini problem}\label{SectionSignorinicasscalairesansu}
In this part we assume that $\Gamma$ is decomposed (up to a null set) as
$$
\Gamma_{\mathrm{N}}\cup\Gamma_{\mathrm{D}}\cup\Gamma_{S-}\cup\Gamma_{S+},
$$
where $\Gamma_{\mathrm{N}}$, $\Gamma_{\mathrm{D}}$, $\Gamma_{\mathrm{S-}}$ and $\Gamma_{\mathrm{S+}}$ are four measurable pairwise disjoint subsets of $\Gamma$. Consider the scalar Signorini problem given by
\begin{equation} \label{PbDirichletNeumannSigno}\tag{SP}
\arraycolsep=2pt
\left\{
\begin{array}{rcll}
-\Delta u +u & = &  k   & \text{ in } \Omega , \\
u & = & w  & \text{ on } \Gamma_{\mathrm{D}}, \\
\partial_{\nn} u & = & \ell  & \text{ on } \Gamma_{\mathrm{N}}, \\
 u\leq w\text{, } \partial_{\nn}u\leq \ell \text{ and} \left(u-w\right)\left(\partial_{\nn}u-\ell\right)	& = & 0  & \text{ on } \Gamma_{\mathrm{S-}}, \\
 u\geq w\text{, } \partial_{\nn}u\geq \ell \text{ and} \left(u-w\right)\left(\partial_{\nn}u-\ell\right)	& = & 0  & \text{ on } \Gamma_{\mathrm{S+}},
\end{array}
\right.
\end{equation}
where the data have been introduced at the beginning of Subsection~\ref{boundary}.

\begin{myDefn}[Solution to the scalar Signorini problem]
  A (strong) solution to the scalar Signorini problem~\eqref{PbDirichletNeumannSigno} is a function $u\in \HH^{1}(\Omega)$ such that $-\Delta u+u=f$ in $\mathcal{C}^{\infty}_0(\Omega)'$, $u=w$ \textit{a.e.} on~$\Gamma_{\mathrm{D}}$, and also~$\partial_{\nn}u\in \mathrm{L}^{2}(\Gamma)$ with $\partial_{\nn}u=\ell$ \textit{a.e.} on $\Gamma_{\mathrm{N}}$, $u\leq w$, $\partial_{\nn}u\leq \ell$ and $(u-w)(\partial_{\nn}u-\ell)=0$ \textit{a.e.} on $\Gamma_{\mathrm{S-}}$, $u\geq w$, $\partial_{\nn}u\geq \ell$ and~$(u-w)(\partial_{\nn}u-\ell)=0$ \textit{a.e.} on $\Gamma_{\mathrm{S+}}$.
\end{myDefn}

\begin{myDefn}[Weak solution to the scalar Signorini problem]\label{WeaksolutionSigno}
A weak solution to the scalar Signorini problem~\eqref{PbDirichletNeumannSigno} is a function $u\in\mathcal{K}^{1}_{w}(\Omega)$ such that 
\begin{equation*}
  \displaystyle\int_{\Omega}\nabla{u}\cdot\nabla(v-u)+\int_{\Omega}u(v-u)\geq\int_{\Omega}k(v-u)+\int_{\Gamma}\ell(v-u),\qquad \forall v\in\mathcal{K}^{1}_{w}(\Omega),
\end{equation*}
where $\mathcal{K}^{1}_{w}(\Omega)$ is the nonempty closed convex subset of $\HH^{1}(\Omega)$ defined by
$$
\mathcal{K}^{1}_{w}(\Omega) := \left\{v\in\HH^{1}(\Omega) \mid v\leq w \text{ \textit{a.e.} on } \Gamma_{\mathrm{S-}}\text{, } v=w \text{ \textit{a.e.} on }\Gamma_{\mathrm{D}} \text{ and } v\geq w \text{ \textit{a.e.} on } \Gamma_{\mathrm{S+}} \right \}.
$$
\end{myDefn}

One can easily prove that a (strong) solution to the scalar Signorini problem~\eqref{PbDirichletNeumannSigno} is also a weak solution. However, to the best of our knowledge, one cannot prove the converse without additional assumptions. To get the equivalence, one can assume, in particular, that the decomposition $\Gamma_{\mathrm{N}}\cup\Gamma_{\mathrm{D}}\cup\Gamma_{S-}\cup\Gamma_{S+}$ is \textit{consistent} in the following sense.

\begin{myDefn}[Consistent decomposition]\label{regulieresens2}
 The decomposition   $\Gamma_{\mathrm{N}}\cup\Gamma_{\mathrm{D}}\cup\Gamma_{S-}\cup\Gamma_{S+}$ is said to be \textit{consistent} if
 \begin{enumerate}[label={\rm (\roman*)}]
     \item For almost all $s\in\Gamma_{\mathrm{S-}}$ (resp.\ $\Gamma_{\mathrm{S+}}$),  $s\in\mathrm{int}_{\Gamma}(\Gamma_{{\mathrm{S-}}})$ (resp.\ $s\in\mathrm{int}_{\Gamma}(\Gamma_{{\mathrm{S+}}})$), where the notation~$\mathrm{int}_{\Gamma}$ stands for the interior relative to $\Gamma$;
     \item The nonempty closed convex subset $\mathcal{K}_{w}^{1/2}(\Gamma)$ of $\HH^{1/2}(\Gamma)$ defined by
    $$         
    \hspace{-0.2cm}\mathcal{K}_{w}^{1/2}(\Gamma):=\left \{ v\in \HH^{1/2}(\Gamma) \mid v\leq w \text{ \textit{a.e.} on } \Gamma_{\mathrm{S-}}\text{, } v=w \text{ \textit{a.e.} on }\Gamma_{\mathrm{D}} \text{ and } v\geq w \text{ \textit{a.e.}on } \Gamma_{\mathrm{S+}} \right \},
    $$
is dense in  the nonempty closed convex subset $\mathcal{K}_{w}^{0}(\Gamma)$ of $\mathrm{L}^{2}(\Gamma)$ defined by
     $$ \mathcal{K}_{w}^{0}(\Gamma):=\left \{ v\in \mathrm{L}^{2}(\Gamma) \mid v\leq w \text{ \textit{a.e.} on } \Gamma_{\mathrm{S-}}\text{, } v=w \text{ \textit{a.e.} on }\Gamma_{\mathrm{D}} \text{ and } v\geq w \text{ \textit{a.e.} on } \Gamma_{\mathrm{S+}} \right \}.
     $$
\end{enumerate}
\end{myDefn}

\begin{myProp}\label{EquiSignoCassansu}
Let $u\in \HH^{1}(\Omega)$.
\begin{enumerate}[label={\rm (\roman*)}]
    \item If $u$ is a (strong) solution to the scalar Signorini problem~\eqref{PbDirichletNeumannSigno}, then $u$ is a weak solution to the scalar Signorini problem~\eqref{PbDirichletNeumannSigno}.
    \item If $u$ is a weak solution to the scalar Signorini problem~\eqref{PbDirichletNeumannSigno} such that $\partial_{\nn}u\in \mathrm{L}^{2}(\Gamma)$ and the decomposition $\Gamma_{\mathrm{N}}\cup\Gamma_{\mathrm{D}}\cup\Gamma_{S-}\cup\Gamma_{S+}$ is consistent, then $u$ is a (strong) solution to the scalar Signorini problem~\eqref{PbDirichletNeumannSigno}.
\end{enumerate}
\end{myProp}

Using the classical characterization of the projection operator, one can easily get the following existence/uniqueness result.

\begin{myProp}\label{existenceunicitepbDNS}
The scalar Signorini problem~\eqref{PbDirichletNeumannSigno} admits a unique weak solution $u\in\HH^{1}(\Omega)$ characterized by
$$
    \displaystyle u=\mathrm{proj}_{\mathcal{K}_{w}^{1}(\Omega)}(F),
$$
where $F\in\HH^{1}(\Omega)$ is the unique solution to the Neumann problem
$$
\arraycolsep=2pt
\left\{
\begin{array}{rcll}
-\Delta F+F  & = & k   & \text{ in } \Omega , \\
\partial_{\nn}F & = & \ell  & \text{ on } \Gamma,\\
\end{array}
\right.
$$ 
and where $\mathrm{proj}_{\mathcal{K}_{w}^{1}(\Omega)} : \HH^{1}(\Omega) \to \HH^{1}(\Omega)$ stands for the classical projection operator onto the nonempty closed convex subset $\mathcal{K}_{w}^{1}(\Omega)$ of $\HH^{1}(\Omega)$ for the usual scalar product~$\dual{\cdot}{\cdot}_{\HH^{1}(\Omega)}$.
\end{myProp}

\subsubsection{A scalar Tresca friction problem}\label{sectionproblèmedeTresca1ercas}
In this part we assume that $\ell>0$ \textit{a.e.} on $\Gamma$. Consider the scalar Tresca friction problem given by
\begin{equation}\label{Trescageneral}\tag{TP}
\arraycolsep=2pt
\left\{
\begin{array}{rcll}
-\mathrm{div}(\mathrm{M}\nabla{u})+uh  & = & k   & \text{ in } \Omega , \\
\left|\mathrm{M}\nabla{u}\cdot\boldsymbol{\nn}\right|\leq \ell \text{ and }  u\:\mathrm{M}\nabla{u}\cdot \boldsymbol{\nn}+\ell\left|u\right|&=&0  & \text{ on } \Gamma,\end{array}
\right.
\end{equation}
where the data have been introduced at the beginning of Subsection~\ref{boundary}.

\begin{myDefn}[Solution to the scalar Tresca friction problem]
A (strong) solution to the scalar Tresca friction problem~\eqref{Trescageneral} is a function $u\in \HH^{1}(\Omega)$ such that $-\mathrm{div}(\mathrm{M}\nabla{u})+uh=k$ in~$\mathcal{C}^{\infty}_0(\Omega)'$, $\mathrm{M}\nabla{u}\cdot\boldsymbol{\nn}\in\mathrm{L}^{2}(\Gamma)$ with $|\mathrm{M}(s)\nabla{u}(s)\cdot\boldsymbol{\nn(s)}|\leq \ell(s)$ and $u(s)\mathrm{M}(s)\nabla{u}(s)\cdot\boldsymbol{\nn(s)}+\ell(s)|u(s)|=0$ for almost all $s\in \Gamma$.
\end{myDefn}

\begin{myDefn}[Weak solution to the scalar Tresca friction problem]\label{WFTresca}
A weak solution to the scalar Tresca friction problem~\eqref{Trescageneral} is a function $u\in \HH^{1}(\Omega)$ such that
 \begin{equation*}
     \int_{\Omega} \mathrm{M}\nabla u\cdot\nabla (v-u)+\int_{\Omega} uh(v-u)+\int_{\Gamma}\ell|v|-\int_{\Gamma}\ell|u|\geq \int_{\Omega}k(v-u), \qquad \forall v\in\HH^{1}(\Omega).
\end{equation*}
\end{myDefn}

\begin{myProp}\label{Trescaequivalartc}
A function $u\in \HH^{1}(\Omega)$ is a (strong) solution to the scalar Tresca friction problem~\eqref{Trescageneral} if and only if $u$ is a weak solution to the scalar Tresca friction problem~\eqref{Trescageneral}.
\end{myProp}

Using the classical characterization of the proximal operator, we obtain the following existence/uniqueness result.

\begin{myProp}\label{existenceunicitePbDNT}
The scalar Tresca friction problem~\eqref{Trescageneral} admits a unique (strong) solution~$u\in\HH^{1}(\Omega)$ characterized by
$$
\displaystyle u=\mathrm{prox}_{\phi}(F),
$$
where $F\in\HH^{1}(\Omega)$ is the unique solution to the Neumann problem
$$
\arraycolsep=2pt
\left\{
\begin{array}{rcll}
-\mathrm{div}(\mathrm{M}\nabla{F})+Fh  & = & k   & \text{ in } \Omega, \\
\mathrm{M}\nabla{F}\cdot\boldsymbol{\nn} & = & 0  & \text{ on } \Gamma,\\
\end{array}
\right.
$$
and where $\mathrm{prox}_{\phi} : \HH^{1}(\Omega) \to \HH^{1}(\Omega)$ stands for the proximal operator associated with the Tresca friction functional given by
\begin{equation*}
\displaystyle\fonction{\phi}{\HH^{1}(\Omega)}{\R}{v}{\displaystyle \phi(v):=\int_{\Gamma}\ell|v|,}
\end{equation*}
considered on the Hilbert space $(\HH^{1}(\Omega),\dual{\cdot}{\cdot}_{\mathrm{M},h})$.
\end{myProp}

\section{Main theoretical results}\label{mainresult}
Let $d\in\mathbb{N}^{*}$ be a positive integer and let~$f\in\HH^{1}(\R^{d})$ and~$g\in\HH^{2}(\R^{d})$ be such that $g>0$ \textit{a.e.} on~$\R^{d}$. In this paper we consider the shape optimization problem given by
\begin{equation*}
    \minn\limits_{ \substack{ \Omega\in \mathcal{U} \\ \vert \Omega \vert = \lambda } } \; \mathcal{J}(\Omega),
\end{equation*}
where
\begin{equation*}
        \mathcal{U} :=\{ \Omega\subset\R^{d} \mid \Omega \text{ nonempty connected bounded open subset of } \R^{d} \text{ with Lipschitz boundary} \},
\end{equation*}
with the volume constraint $\vert \Omega \vert = \lambda > 0$, where $\mathcal{J} : \mathcal{U} \to \R$ is the \textit{Tresca energy functional} defined by
\begin{equation*}
    \mathcal{J}(\Omega) := \frac{1}{2}\int_{\Omega}\left( \left\| \nabla{u_\Omega} \right\|^{2}+|u_{\Omega}|^{2}\right)+\int_{\Gamma}g|u_{\Omega}|-\int_{\Omega}fu_{\Omega},
\end{equation*}
where $\Gamma:=\partial{\Omega}$ is the boundary of $\Omega$ and where $u_\Omega \in\HH^{1}(\Omega)$ stands for the unique solution to the scalar Tresca friction problem given by
\begin{equation}\label{Trescaproblem222}\tag{TP\ensuremath{{}_\Omega}}
\arraycolsep=2pt
\left\{
\begin{array}{rcll}
-\Delta u+u &=&f    & \text{ in } \Omega, \\
|\partial_{\nn}u|\leq g \text{ and }  u\partial_{\nn}u+g|u|&=&0  & \text{ on } \Gamma,
\end{array}
\right.
\end{equation}
for all~$\Omega \in \mathcal{U}$. From Subsection~\ref{sectionproblèmedeTresca1ercas}, note that $\mathcal{J}$ can also be expressed as
$$
  \mathcal{J}(\Omega) = -\frac{1}{2}\int_{\Omega} \left(\left\| \nabla{u_\Omega} \right\|^{2}+|u_\Omega|^{2}\right),
$$
for all~$\Omega \in \mathcal{U}$.

In the whole section let us fix~$\Omega_0 \in \mathcal{U}$. We denote by $\boldsymbol{\mathrm{id}} :  \R^{d}\rightarrow \R^{d}$ the identity operator. Our aim here is to prove that, under appropriate assumptions, the functional $\mathcal{J}$ is \textit{shape differentiable} at $\Omega_0$, in the sense that the map
$$
\begin{array}{rcl}
\mathcal{C}^{1,\infty}(\R^{d},\R^{d}) & \longrightarrow & \R \\
\boldsymbol{V} & \longmapsto & \displaystyle \mathcal{J}((\boldsymbol{\mathrm{id }}+\boldsymbol{V})(\Omega_0)),
\end{array}
$$
where $\mathcal{C}^{1,\infty}(\R^{d},\R^{d}):=\mathcal{C}^{1}(\R^{d},\R^{d})\cap\mathrm{W}^{1,\infty}(\R^{d},\R^{d}) $, is Gateaux differentiable at~$0$, and to give an expression of the Gateaux differential, denoted by~$\mathcal{J}'(\Omega_0)$, which is called the \textit{shape gradient} of~$\mathcal{J}$ at~$\Omega_0$. To this aim we have to perform the sensitivity analysis of the scalar Tresca friction problem~\eqref{Trescaproblem222} with respect to the shape, and then characterize the material and shape directional derivatives. 

For better organization, this part will be done in the following three separate subsections below. In Subsection~\ref{PerturbTrescatyoefric}, we perturb the scalar Tresca friction problem~(TP${}_{\Omega_{0}}$) with respect to the shape. In Subsection~\ref{sectshape}, under appropriate assumptions, we characterize the material directional derivative as solution to a variational inequality (see Theorem~\ref{Lagderiv}). Additionally, assuming a regularity assumption on the solution to the scalar Tresca friction problem, we characterize the material and shape directional derivatives as being weak solutions to scalar Signorini problems (see Corollaries~\ref{LagderivSigno} and~\ref{shapederivaticesigno}). Finally we prove in Subsection~\ref{energyfunctional} our main result asserting that, under appropriate assumptions, the functional~$\mathcal{J}$ is shape differentiable at~$\Omega_0$ and we provide an expression of the shape gradient~$\mathcal{J}'(\Omega_0)$ (see Theorem~\ref{shapederivofJ} and Corollary~\ref{shapederivofJ3}).

\subsection{Setting of the shape perturbation and preliminaries}\label{PerturbTrescatyoefric}
 Consider $\boldsymbol{V}\in \mathcal{C}^{1,\infty}(\R^{d},\R^{d})$ and, for all $t\geq0$ sufficiently small such that~$\boldsymbol{\mathrm{id}}+t\boldsymbol{V}$ is a $\mathcal{C}^{1}$-diffeomorphism of $\R^{d}$, consider the shape perturbed scalar Tresca friction problem given by
\begin{equation}\label{Trescaproblem}\tag{TP\ensuremath{_{t}}}
\arraycolsep=2pt
\left\{
\begin{array}{rcll}
-\Delta u_{t}+u_{t} &=&f    & \text{ in } \Omega_{t} , \\
|\partial_{\nn}u_{t}|\leq g \text{ and }  u_{t}\partial_{\nn}u_{t}+g|u_{t}|&=&0  & \text{ on } \Gamma_{t},
\end{array}
\right.
\end{equation}
where $\Omega_{t}:=(\boldsymbol{\mathrm{id}}+t\boldsymbol{V})(\Omega_{0})\in\mathcal{U}$ and $\Gamma_{t}:= \partial \Omega_t = (\boldsymbol{\mathrm{id}}+t\boldsymbol{V})(\Gamma_{0})$.
From Subsection~\ref{sectionproblèmedeTresca1ercas}, there exists a unique solution $u_{t}\in\HH^{1}(\Omega_{t})$ to~\eqref{Trescaproblem} which satisfies 
\begin{equation*}
     \int_{\Omega_{t}} \nabla u_{t}\cdot\nabla (v-u_{t})+\int_{\Omega_{t}} u_{t} (v-u_{t})+\int_{\Gamma_{t}}g|v|-\int_{\Gamma_{t}}g|u_{t}|\geq \int_{\Omega_{t}}f(v-u_{t}), \qquad\forall v\in\HH^{1}(\Omega_{t}).
\end{equation*}
Following the usual strategy in shape optimization literature (see, e.g., \cite{HENROT}) and using the change of variables $\boldsymbol{\mathrm{id}}+t\boldsymbol{V}$, we prove that $\overline{u}_t:=u_{t}\circ(\boldsymbol{\mathrm{id}}+t\boldsymbol{V})\in\HH^{1}(\Omega_{0})$ satisfies
\begin{multline*}
\int_{\Omega_{0}} \mathrm{A}_{t}\nabla \overline{u}_t\cdot\nabla (v-\overline{u}_t)+\int_{\Omega_{0}} \overline{u}_t(v-\overline{u}_t)\mathrm {J}_{t}+\int_{\Gamma_{0}}g_{t}\mathrm {J}_{\mathrm{T}_{t}}|v|-\int_{\Gamma_{0}}g_{t}\mathrm {J}_{\mathrm{T}_{t}}|\overline{u}_t|\\\geq \int_{\Omega_{0}}f_{t}\mathrm {J}_{t}(v-\overline{u}_t), \qquad \forall v\in\HH^{1}(\Omega_{0}),
\end{multline*}
where $f_{t}:=f\circ(\boldsymbol{\mathrm{id}}+t\boldsymbol{V})\in\HH^{1}(\R^{d})$, $g_{t}:=g\circ(\boldsymbol{\mathrm{id}}+t\boldsymbol{V})\in\HH^{2}(\R^{d})$,
 $\mathrm{J}_{t}:=\mathrm{det}(\mathrm{I}+t\nabla{\boldsymbol{V}})\in\LL^{\infty}(\R^{d})$ is the Jacobian determinant, $\mathrm{A}_{t}:=\mathrm{det}(\mathrm{I}+t\nabla{\boldsymbol{V}})(\mathrm{I}+t\nabla{\boldsymbol{V}})^{-1}(\mathrm{I}+t\nabla{\boldsymbol{V}}^{\top})^{-1}\in\LL^{\infty}(\R^{d},\R^{d \times d})$ and~$\mathrm {J}_{\mathrm{T}_{t}}:=\mathrm{det}(\mathrm{I}+t\nabla{\boldsymbol{V}}) \|(\mathrm{I}+t\nabla{\boldsymbol{V}}^{\top})^{-1}\boldsymbol{\nn} \| \in\mathcal{C}^{0}(\Gamma_{0})$ is the tangential Jacobian,
where~$\mathrm{I} $ stands for the identity matrix of~$\R^{d \times d}$. Therefore, we deduce from Subsection~\ref{sectionproblèmedeTresca1ercas} that~$\overline{u}_t\in\HH^{1}(\Omega_{0})$ is the unique solution to the perturbed scalar Tresca friction problem
\begin{equation}\label{Trescashape}\tag{\ensuremath{\overline{\mathrm{TP}}_{t}}}
\arraycolsep=2pt
\left\{
\begin{array}{rcll}
-\mathrm{div}\left(\mathrm{A}_{t}\nabla\overline{u}_t\right)+\overline{u}_t\mathrm{J}_{t} &=&f_{t}\mathrm{J}_{t}    & \text{ in } \Omega_{0} , \\
\left|\mathrm{A}_{t}\nabla\overline{u}_t\cdot\boldsymbol{\nn}\right|\leq g_{t}\mathrm {J}_{\mathrm{T}_{t}}\text{ and }  \overline{u}_t\mathrm{A}_{t}\nabla\overline{u}_t\cdot\boldsymbol{\nn}+g_{t}\mathrm {J}_{\mathrm{T}_{t}}\left|\overline{u}_t\right|&=&0  & \text{ on } \Gamma_{0},
\end{array}
\right.
\end{equation}
and can be expressed as
$$
   \displaystyle \overline{u}_t=\mathrm{prox}_{\phi_{t}}(F_{t}),
$$
where $F_{t}\in\HH^{1}(\Omega_{0})$ is the unique solution to the perturbed Neumann problem
\begin{equation*}
\arraycolsep=2pt
\left\{
\begin{array}{rcll}
-\mathrm{div}\left(\mathrm{A}_{t}\nabla F_{t}\right)+F_{t}\mathrm{J}_{t} &=&f_{t}\mathrm{J}_{t}    & \text{ in } \Omega_{0} , \\
\mathrm{A}_{t}\nabla F_{t}\cdot \boldsymbol{\nn}&=&0  & \text{ on } \Gamma_{0},
\end{array}
\right.
\end{equation*}
and $\mathrm{prox}_{\phi_{t}} : \HH^{1}(\Omega_{0}) \to \HH^{1}(\Omega_{0})$ is the proximal operator associated with the perturbed Tresca friction functional
$$
\displaystyle\fonction{\phi_{t}}{\HH^{1}(\Omega_{0})}{\R}{v}{\displaystyle \phi_{t}(v):=\int_{\Gamma_{0}}g_{t}\mathrm {J}_{\mathrm{T}_{t}}|v|,}
$$
considered on the perturbed Hilbert space $(\HH^{1}(\Omega_{0}),\dual{\cdot}{\cdot}_{\mathrm{A}_{t},\mathrm{J}_{t}})$.

Since the derivative of the map~$ t \in \R_+ \mapsto F_t \in \HH^{1}(\Omega_{0})$ at~$t=0$ is well known in the literature (it can be proved in a similar way as in Lemma~\ref{lem876} below), one might believe that Proposition~\ref{TheoABC2018} could allow to compute the derivative of the map~$ t \in \R_+ \mapsto \overline{u}_t \in \HH^{1}(\Omega_{0})$ at~$t=0$ (that is, the material directional derivative) under the assumption of the twice epi-differentiability of the parameterized functional~$\phi_t$. This would be very similar to the strategy developed in our previous paper~\cite{BCJDC} in which we have considered a simpler case where~$\mathrm{J}_{t}=\mathrm{J}_{\mathrm{T}_{t}}=1$ and $\mathrm{A_{t}}=\mathrm{I}$ and where, therefore, the scalar product~$\dual{\cdot}{\cdot}_{\mathrm{A}_{t},\mathrm{J}_{t}}$ was independent of~$t$. However, in the present work, we face a scalar product $\dual{\cdot}{\cdot}_{\mathrm{A}_{t},\mathrm{J}_{t}}$ that is~$t$-dependent and we need to overcome this difficulty as follows. Let us write $\mathrm{A}_{t}=\mathrm{I}+(\mathrm{A}_{t}-\mathrm{I})$ and~$\mathrm{J}_{t}=1+(\mathrm{J}_{t}-1)$ to get
\begin{multline*}
\dual{\overline{u}_t}{v-\overline{u}_t}_{\HH^{1}(\Omega_{0})}+\int_{\Gamma_{0}}g_{t}\mathrm {J}_{\mathrm{T}_{t}}|v|-\int_{\Gamma_{0}}g_{t}\mathrm {J}_{\mathrm{T}_{t}}|\overline{u}_t|\geq \int_{\Omega_{0}}f_{t}\mathrm {J}_{t}(v-\overline{u}_t)\\-\int_{\Omega_{0}}\left(\mathrm{A}_{t}-\mathrm{I}\right)\nabla{\overline{u}_t}\cdot\nabla{\left(v-\overline{u}_t\right)}-\int_{\Omega_{0}}\left(\mathrm{J}_{t}-1\right)\overline{u}_t\left(v-\overline{u}_t\right), \qquad \forall v\in\HH^{1}(\Omega_{0}),
\end{multline*}
and thus
\begin{equation*}
     \displaystyle \overline{u}_t=\mathrm{prox}_{\Phi(t,\cdot)}(E_{t}),
\end{equation*}
where $E_{t}\in\HH^{1}(\Omega_{0})$ stands for the unique solution to the perturbed variational Neumann problem given by
\begin{equation*}
    \dual{E_{t}}{v}_{\HH^{1}(\Omega_{0})}=\int_{\Omega_{0}}f_{t}\mathrm {J}_{t}v-\int_{\Omega_{0}}\left(\mathrm{A}_{t}-\mathrm{I}\right)\nabla{\overline{u}_t}\cdot\nabla{v}-\int_{\Omega_{0}}\left(\mathrm{J}_{t}-1\right)\overline{u}_tv, \qquad \forall v\in\HH^{1}(\Omega_{0}),
\end{equation*}
and where $\mathrm{prox}_{\Phi(t,\cdot)} : \HH^{1}(\Omega_{0}) \to \HH^{1}(\Omega_{0})$ is the proximal operator associated with the parameterized Tresca friction functional defined by
\begin{equation*}
    \displaystyle\fonction{\Phi}{\mathbb{R}_{+}\times\HH^{1}(\Omega_{0})}{\R}{(t,v)}{\displaystyle \Phi(t,v):=\int_{\Gamma_{0}}g_{t}\mathrm {J}_{\mathrm{T}_{t}}|v|,}
\end{equation*}
considered on the standard Hilbert space $(\HH^{1}(\Omega_{0}),\dual{\cdot}{\cdot}_{\HH^{1}(\Omega_{0})})$ whose scalar product is the usual~$t$-independent one.

\begin{myRem}\normalfont
Note that the existence/uniqueness of the solution $E_{t}\in\HH^{1}(\Omega_{0})$ to the above perturbed variational Neumann problem can be easily derived from the Riesz representation theorem. Furthermore note that, if $\mathrm{div}\left(\left(\mathrm{A}_{t}-\mathrm{I}\right)\nabla\overline{u}_t\right)\in\LL^{2}(\Omega_{0})$, then the above perturbed variational Neumann problem corresponds exactly to the weak variational formulation of the perturbed Neumann problem given by
\begin{equation*}
\arraycolsep=2pt
\left\{
\begin{array}{rcll}
-\Delta E_{t}+E_{t}  & = & f_{t}\mathrm {J}_{t}-\left(\mathrm{J}_{t}-1\right)\overline{u}_t+\mathrm{div}\left(\left(\mathrm{A}_{t}-\mathrm{I}\right)\nabla\overline{u}_t\right) & \text{ in } \Omega_{0} , \\
\partial_{\nn}E_{t} & = &  -\left(\mathrm{A}_{t}-\mathrm{I}\right)\nabla\overline{u}_t\cdot\boldsymbol{\nn} & \text{ on } \Gamma_{0}.\\
\end{array}
\right.
\end{equation*}
For instance, note that the condition $\mathrm{div}\left(\left(\mathrm{A}_{t}-\mathrm{I}\right)\nabla\overline{u}_t\right)\in\LL^{2}(\Omega_{0})$ is satisfied when $\overline{u}_t\in\HH^{2}(\Omega_{0})$.
\end{myRem}

Now our next step is to derive the differentiability of the map $t\in\R_{+} \mapsto E_{t} \in \HH^{1}(\Omega_{0})$ at $t=0$. To this aim let us recall that (see~\cite{HENROT}):
\begin{enumerate}[label={\rm (\roman*)}]
    \item The map $t\in\R_{+} \mapsto \mathrm{J}_{t}\in\LL^{\infty}(\R^{d})$ is differentiable at $t=0$ with derivative given by~$\mathrm{div}(\boldsymbol{V})$;
    \item The map $t\in\R_{+} \mapsto f_{t}\mathrm{J}_{t}\in\LL^{2}(\R^{d})$ is differentiable at $t=0$ with derivative given by~$f\mathrm{div}(\boldsymbol{V})+\nabla{f}\cdot \boldsymbol{V}$;
    \item The map $t\in\R_{+} \mapsto \mathrm{A}_{t}\in\LL^{\infty}(\R^{d},\R^{d \times d})$ is differentiable at $t=0$ with derivative given by~$\mathrm{A}'_{0}:=-\nabla{\boldsymbol{V}}-\nabla{\boldsymbol{V}}^{\top}+\mathrm{div}(\boldsymbol{V})\mathrm{I}$;
    \item The map $t\in\R_{+}\mapsto g_{t}\mathrm {J}_{\mathrm{T}_{t}}\in\LL^{2}(\Gamma_{0})$ is differentiable at $t=0$ with derivative given by~$\nabla{g}\cdot \boldsymbol{V}+g \mathrm{div}_{\Gamma_0}(\boldsymbol{V})$.
\end{enumerate}

\begin{myLem}\label{lem876}
The map $t\in\R_{+} \mapsto E_{t} \in \HH^{1}(\Omega_{0})$ is differentiable at $t=0$ and its derivative, denoted by $E'_{0} \in \HH^{1}(\Omega_{0})$, is the unique solution to the variational Neumann problem given by
\begin{multline}\label{Woderov}
    \dual{E'_{0}}{v}_{\HH^{1}(\Omega_{0})}=\int_{\Omega_{0}}\left(f\mathrm{div}(\boldsymbol{V})+\nabla{f}\cdot \boldsymbol{V}\right)v\\-\int_{\Omega_{0}}\left(-\nabla{\boldsymbol{V}}-\nabla{\boldsymbol{V}}^{\top}+\mathrm{div}(\boldsymbol{V})\mathrm{I}\right)\nabla{u_{0}}\cdot\nabla{v}-\int_{\Omega_{0}}\mathrm{div}(\boldsymbol{V})u_{0}v, \qquad\forall v\in\HH^{1}(\Omega_{0}).
\end{multline}
\end{myLem}

\begin{proof}
Using the Riesz representation theorem, we denote by $Z \in\HH^{1}(\Omega_{0})$ the unique solution to the above variational Neumann problem. From linearity we get that
\begin{multline*}
    \left\|\frac{E_{t}-E_{0}}{t}-Z\right\|_{\HH^{1}(\Omega_0)}\leq \left\|\frac{f_{t}\mathrm{J}_{t}-f}{t}-\left(f\mathrm{div}(\boldsymbol{V})+\nabla{f}\cdot \boldsymbol{V}\right)\right\|_{\LL^{2}(\R^{d})}\\
    + \left\|\frac{\mathrm{A}_{t}-\mathrm{I}}{t}-\left(-\nabla{\boldsymbol{V}}-\nabla{\boldsymbol{V}}^{\top}+\mathrm{div}(\boldsymbol{V})\mathrm{I}\right)\right\|_{\LL^{\infty}(\R^{d},\R^{d \times d})}\left\|\overline{u}_t\right\|_{\HH^{1}(\Omega_{0})}\\
    + \left\|-\nabla{\boldsymbol{V}}-\nabla{\boldsymbol{V}}^{\top}+\mathrm{div}(\boldsymbol{V})\mathrm{I}\right\|_{\LL^{\infty}(\R^{d},\R^{d \times d})}\left\| \overline{u}_t-u_{0}\right\|_{\HH^{1}(\Omega_{0})} \\ +\left\|\frac{\mathrm{J}_{t}-1}{t}-\mathrm{div}(\boldsymbol{V})\right\|_{\LL^{\infty}(\R^{d})}\left\| \overline{u}_t\right\|_{\HH^{1}(\Omega_{0})} 
    + \left\|\mathrm{div}(\boldsymbol{V})\right\|_{\LL^{\infty}(\R^{d})}\left\| \overline{u}_t-u_{0}\right\|_{\HH^{1}(\Omega_{0})},
\end{multline*}
for all $t>0$. Therefore, to conclude the proof, we only need to prove the continuity of the map~$t\in\R_{+} \mapsto \overline{u}_t\in~\HH^{1}(\Omega_{0})$ at $t=0$. To this aim let us take $v=u_{0}$ in the weak variational formulation of $\overline{u}_t$ and~$v=\overline{u}_t$ in the weak variational formulation of $u_{0}$ to get
\begin{multline*}
    -\left\|\overline{u}_t-u_{0}\right\|_{\HH^{1}(\Omega_{0})}^{2}+\int_{\Omega_{0}}\left(\mathrm{A}_{t}-\mathrm{I}\right)\nabla{\overline{u}_t}\cdot\nabla{\left(u_{0}-\overline{u}_t\right)}\\+\int_{\Omega_{0}}\left(\mathrm{J}_{t}-1\right)\overline{u}_t\left(u_{0}-\overline{u}_t\right)+\int_{\Gamma_{0}}\left(g_{t}\mathrm {J}_{\mathrm{T}_{t}}-g\right)\left(\left|u_{0}\right|-\left|\overline{u}_t\right|\right)\geq\int_{\Omega_{0}}\left(f_{t}\mathrm{J}_{t}-f\right)\left(u_{0}-\overline{u}_t\right),
\end{multline*}
which leads to
\begin{multline*}
    \left\|\overline{u}_t-u_{0}\right\|_{\HH^{1}(\Omega_{0})}\leq\left(\left\|\mathrm{A}_{t}-\mathrm{I}\right\|_{\LL^{\infty}(\R^{d},\R^{d \times d})}+\left\|\mathrm{J}_{t}-1\right\|_{\LL^{\infty}(\R^{d})}\right)\left\|\overline{u}_t\right\|_{\HH^{1}(\Omega_{0})}\\+C\left\| g_{t}\mathrm {J}_{\mathrm{T}_{t}}-g \right\|_{\LL^{2}(\Gamma_{0})}+\left\| f_{t}\mathrm {J}_{t}-f \right\|_{\LL^{2}(\R^{d})},
\end{multline*}
for all $t\geq0$, where $C>0$ is a constant that depends only on $\Omega_{0}$. Therefore, to conclude the proof, we only need to prove that the map $t\in\R_{+} \mapsto\left\|\overline{u}_t\right\|_{\HH^{1}(\Omega_{0})}\in\R$ is bounded for $t\geq0$ sufficiently small. For this purpose, let us take~$v=0$ in the weak variational formulation of~$\overline{u}_t$ to get that
$$
\int_{\Omega_{0}}\mathrm{A}_{t}\nabla{\overline{u}_t}\cdot\nabla{\overline{u}_t}+\int_{\Omega_{0}}|\overline{u}_t|^{2}\mathrm{J}_{t}\leq\int_{\Omega_{0}}f_{t}\mathrm{J}_{t}\overline{u}_t-\int_{\Gamma_{0}}g_{t}\mathrm {J}_{\mathrm{T}_{t}}|\overline{u}_t|,
$$
for all $t\geq0$, and thus
$$
\left\| \overline{u}_t \right\|_{\HH^{1}(\Omega_{0})}\leq 2\left(\left\| f\right\|_{\HH^{1}(\R^{d})}+2\left\| g\right\|_{\HH^{1}(\R^{d})}\right),
$$
for all $t\geq0$ sufficiently small, which concludes the proof.
\end{proof}

\begin{myRem}\normalfont\label{remarqueneumannvar}
Note that, if $\mathrm{div}((-\nabla{\boldsymbol{V}}-\nabla{\boldsymbol{V}}^{\top}+\mathrm{div}(\boldsymbol{V})\mathrm{I})\nabla{u_{0}})\in\LL^{2}(\Omega_{0})$, then the variational Neumann problem in Lemma~\ref{lem876} corresponds exactly to the weak variational formulation of the Neumann problem given by
\begin{equation*}
\arraycolsep=1.5pt
\left\{
\begin{array}{rcll}
-\Delta E'_{0}+E'_{0}  & = & f\mathrm{div}(\boldsymbol{V})+\nabla{f}\cdot \boldsymbol{V}-\mathrm{div}(\boldsymbol{V})u_{0}+\mathrm{div}\left(\left(-\nabla{\boldsymbol{V}}-\nabla{\boldsymbol{V}}^{\top}+\mathrm{div}(\boldsymbol{V})\mathrm{I}\right)\nabla u_{0}\right)   & \text{ in } \Omega_{0} , \\
\partial_{\nn}E'_{0} & = & \left(\nabla{\boldsymbol{V}}+\nabla{\boldsymbol{V}}^{\top}-\mathrm{div}(\boldsymbol{V})\mathrm{I}\right)\nabla u_{0}\cdot\boldsymbol{\nn} & \text{ on } \Gamma_{0}.
\end{array}
\right.
\end{equation*}
For instance, note that the condition $\mathrm{div}((-\nabla{\boldsymbol{V}}-\nabla{\boldsymbol{V}}^{\top}+\mathrm{div}(\boldsymbol{V})\mathrm{I})\nabla{u_{0}})\in\LL^{2}(\Omega_{0})$ is satisfied when~$u_0\in\HH^{2}(\Omega_{0})$ and $\boldsymbol{V}\in\mathcal{C}^{2,\infty}(\R^d,\R^d):=~\mathcal{C}^{2}(\R^{d},\R^{d})\cap\mathrm{W}^{2,\infty}(\R^{d},\R^{d})$.
\end{myRem}

\subsection{Material and shape directional derivatives}\label{sectshape}

Consider the framework of Subsection~\ref{PerturbTrescatyoefric}. In particular recall that~$g \in \mathrm{H}^2(\R^d)$ with~$g > 0$ \textit{a.e.}\ on $\R^d$. Our aim in this subsection is to {characterize} the material directional derivative, that is, the derivative of the map~$t\in\R_{+} \mapsto \overline{u}_{t} \in \HH^{1}(\Omega_{0})$ at~$t=0$, and then to deduce an expression of the shape directional derivative defined by~$u'_0:=\overline{u}'_0-\nabla{u_{0}}\cdot\boldsymbol{V}$ (which roughly corresponds to the derivative of the map~$t\in\R_{+} \mapsto u_{t} \in \HH^{1}(\Omega_t)$ at~$t=0$). 

In the previous Subsection~\ref{PerturbTrescatyoefric}, since we have expressed $\overline{u}_{t}=\mathrm{prox}_{\Phi(t,\cdot)}(E_{t})$ and characterized in Lemma~\ref{lem876} the derivative of the map~$t\in\R_{+} \mapsto E_{t} \in \HH^{1}(\Omega_{0})$ at~$t=0$, our idea is to use Proposition~\ref{TheoABC2018} in order to derive the material directional derivative. To this aim the twice epi-differentiability of the parameterized Tresca friction functional~$\Phi$ has to be investigated as we did in our previous paper~\cite{BCJDC} from which the next two lemmas are extracted.

\begin{myLem}[Second-order difference quotient function of $\Phi$]\label{epidiffoffunctionG}
Consider the framework of Subsection~\ref{PerturbTrescatyoefric}. For all $t>0$, $u\in \HH^{1}(\Omega)$ and~$v\in\partial\Phi(0,\cdot)(u)$, it holds that
\begin{equation}\label{Delta2}
      \displaystyle\Delta_{t}^{2}\Phi(u|v)(w)=\int_{\Gamma_{0}}\Delta_{t}^{2}G(s)(u(s)|\partial_{\nn}v(s))(w(s)) \, \mathrm{d}s,
\end{equation}
for all $w\in \HH^{1}(\Omega)$, where, for almost all $s\in\Gamma_{0}$, $\Delta_{t}^{2}G(s)(u(s)|\partial_{\nn}v(s))$ stands for the second-order difference quotient function of $G(s)$ at $u(s)\in\R$ for $\partial_{\nn}v(s) \in g(s)\partial|\mathord{\cdot} |(u(s))$, with~$G(s)$ defined by
$$ 
\fonction{G(s)}{\mathbb{R}_{+}\times\mathbb{R}}{\R}{(t,x)}{G(s)(t,x):=g_{t}(s)\mathrm {J}_{\mathrm{T}_{t}}(s)|x|.}
$$
\end{myLem}

\begin{myLem}[Second-order epi-derivative of $G(s)$]\label{épidiffgabs}
Consider the framework of Subsection~\ref{PerturbTrescatyoefric} and assume that, for almost all $s\in\Gamma_{0}$, $g$ has a directional derivative at $s$ in any direction. Then, for almost all $s\in\Gamma_{0}$, the map~$G(s)$ is twice epi-differentiable at any~$x\in\mathbb{R}$ and for all $y\in g(s)\partial |\mathord{\cdot} |(x)$ with
$$
\displaystyle \mathrm{D}_{e}^{2}G(s)(x|y)(z)=\iota_{\mathrm{K}_{ x,\frac{y}{g(s)}}}(z)+\left(\nabla{g}(s)\cdot \boldsymbol{V}(s)+g(s)\mathrm{div}_{\Gamma_0}(\boldsymbol{V})(s)\right)\frac{y}{g(s)}z,
$$
for all $z\in\mathbb{R}$, where $\iota_{\mathrm{K}_{ x,\frac{y}{g(s)}}}$ stands for the indicator function of the nonempty closed convex subset~$\mathrm{K}_{ x,\frac{y}{g(s)}}$ of $\R$ (see Example~\ref{epidiffabs}).
\end{myLem}

We are now in a position to derive our first main result.

\begin{myTheorem}[Material directional derivative]\label{Lagderiv}
Consider the framework of Subsection~\ref{PerturbTrescatyoefric} and assume that:
\begin{enumerate}[label={\rm (\roman*)}]
    \item For almost all $s\in\Gamma_{0}$, $g$ has a directional derivative at $s$ in any direction.
    \item $\Phi$ is twice epi-differentiable at $u_{0}$ for $E_{0}-u_{0}\in\partial \Phi(0,\cdot)(u_{0})$ with
\begin{equation}\label{hypoth1}
\displaystyle\mathrm{D}_{e}^{2}\Phi(u_{0}|E_{0}-u_{0})(w)=\int_{\Gamma_{0}}\mathrm{D}_{e}^{2}G(s)(u_{0}(s)|\partial_{\nn}(E_{0}-u_{0})(s))(w(s)) \, \mathrm{d}s,
\end{equation}    
for all $w\in \HH^{1}(\Omega)$.\label{hypo4}
\end{enumerate}
Then the map $t\in\R_{+} \mapsto \overline{u}_t\in\HH^{1}(\Omega_{0})$ is differentiable at $t=0$, and its derivative (that is, the material directional derivative), denoted by $\overline{u}'_0\in\HH^{1}(\Omega_{0})$, is the unique solution to the variational inequality
\begin{multline}\label{matderivgeneral}
        \dual{\overline{u}'_0}{v-\overline{u}'_0}_{\HH^{1}(\Omega_{0})} \geq\int_{\Omega_{0}}\boldsymbol{V}\cdot\nabla{u_0}\left(v-\overline{u}'_0\right) \\
    -\int_{\Omega_{0}}\left(\left(-\nabla{\boldsymbol{V}}-\nabla{\boldsymbol{V}}^{\top}+\mathrm{div}(\boldsymbol{V})\mathrm{I}\right)\nabla{u_{0}}-\Delta u_0 \boldsymbol{V}\right)\cdot\nabla(v-\overline{u}'_0)\\
    +\int_{\Gamma_{0}}\left(\boldsymbol{V}\cdot \boldsymbol{\nn} \left(f-u_0\right)+\left(\frac{\nabla{g}}{g}\cdot\boldsymbol{V}+\mathrm{div}_{\Gamma_0}(\boldsymbol{V})\right)\partial_{\nn}u_{0}\right)\left(v-\overline{u}'_0\right), \qquad \forall v \in\mathcal{K}_{u_{0},\frac{\partial_{\nn}\left(E_{0}-u_{0}\right)}{g}},
\end{multline}
where $\mathcal{K}_{u_{0},\frac{\partial_{\nn}\left(E_{0}-u_{0}\right)}{g}}$ is the nonempty closed convex subset of $\HH^{1}(\Omega_0)$ defined by
$$
\mathcal{K}_{u_{0},\frac{\partial_{\nn}\left(E_{0}-u_{0}\right)}{g}} := \left\{ v\in \HH^{1}(\Omega_{0})\mid v\leq 0 \text{ \textit{a.e.} on }\Gamma^{u_{0},g}_{\mathrm{S-}}\text{, } v\geq 0 \text{ \textit{a.e.} on }\Gamma^{u_{0},g}_{\mathrm{S+}}\text{, } v=0 \text{ \textit{a.e.} on } \Gamma^{u_{0},g}_{\mathrm{D}} \right\},
$$
where $\Gamma_0$ is decomposed, up to a null set, as~$\Gamma^{u_{0},g}_{\mathrm{N}}\cup
\Gamma^{u_{0},g}_{\mathrm{D}}\cup\Gamma^{u_{0},g}_{\mathrm{S-}}\cup\Gamma^{u_{0},g}_{\mathrm{S+}}$, where
$$
\begin{array}{l}
\Gamma^{u_{0},g}_{\mathrm{N}}:=\left\{s\in\Gamma_{0} \mid  u_{0}(s)\neq0\right \}, \\
\Gamma^{u_{0},g}_{\mathrm{D}}:=\left\{s\in\Gamma_{0} \mid  u_{0}(s)=0 \text{ and } \partial_{\nn}u_{0}(s)\in\left(-g(s),g(s)\right)\right \}, \\
\Gamma^{u_{0},g}_{\mathrm{S-}}:=\left\{s\in\Gamma_{0} \mid u_{0}(s)=0 \text{ and } \partial_{\nn}u_{0}(s)=g(s)\right \}, \\
\Gamma^{u_{0},g}_{\mathrm{S+}}:=\left\{s\in\Gamma_{0} \mid  u_{0}(s)=0 \text{ and } \partial_{\nn}u_{0}(s)=-g(s)\right \}.
\end{array}
$$
\end{myTheorem}

\begin{proof}
The proof is almost identical to~\cite[Theorem~3.21 p.19]{BCJDC}. From Hypothesis~\ref{hypo4} and Lemma~\ref{épidiffgabs}, it follows that
\begin{multline*}
\displaystyle \mathrm{D}_{e}^{2}\Phi(u_{0}|E_{0}-u_{0})(w)=\iota_{\mathcal{K}_{u_{0},\frac{\partial_{\nn}(E_{0}-u_{0})}{g}}}(w)\\+\int_{\Gamma_{0}}\left(\nabla{g}(s)\cdot \boldsymbol{V}(s)+g(s)\mathrm{div}_{\Gamma_0}(\boldsymbol{V})(s)\right)\frac{\partial_{\nn}(E_{0}-u_{0})(s)}{g(s)}w(s)\mathrm{d}s,
\end{multline*}
for all $w\in \HH^{1}(\Omega_{0})$, where $\mathcal{K}_{u_{0},\frac{\partial_{\nn}(E_{0}-u_{0})}{g}}$ is the nonempty closed convex subset of $\HH^{1}(\Omega_{0})$ defined by 
$$
\mathcal{K}_{u_{0},\frac{\partial_{\nn}(E_{0}-u_{0})}{g}}:=\left\{ w\in \HH^{1}(\Omega_{0})\mid w(s)\in \mathrm{K}_{u_{0}(s),\frac{\partial_{\nn}(E_{0}-u_{0})(s)}{\scriptstyle{g(s)}}} \text{ for almost all }s\in\Gamma_{0} \right\}, 
$$
which coincides with the definition given in Theorem~\ref{Lagderiv}. Moreover $\mathrm{D}_{e}^{2}\Phi(u_{0}|E_{0}-u_{0})$ is a proper lower semi-continuous convex function on $\HH^{1}(\Omega_{0})$, and from Lemma~\ref{lem876}, the map $t\in\mathbb{R}^{+}\mapsto E_{t}\in \HH^{1}(\Omega_{0})$ is differentiable at $t=0$, with its derivative~$E'_{0}\in \HH^{1}(\Omega_{0})$ being the unique solution to the variational Neumann problem~\eqref{Woderov}. Thus, using Theorem~\ref{TheoABC2018}, the map $t\in\mathbb{R}^{+}\mapsto \overline{u}_{t}\in \HH^{1}(\Omega_{0})$ is differentiable at $t=0$, and its derivative~$\overline{u}'_0\in\HH^{1}(\Omega_{0})$ satisfies
$$
\overline{u}'_0=\mathrm{prox}_{\mathrm{D}_{e}^{2}\Phi(u_{0}|E_{0}-u_{0})}(E_{0}').
$$
From the definition of the proximal operator (see Definition~\ref{proxi}), this leads to
$$
\displaystyle \dual{ E_{0}'-\overline{u}'_0}{v-\overline{u}'_0}_{\HH^{1}(\Omega_{0})}\leq \mathrm{D}_{e}^{2}\Phi(u_{0}|E_{0}-u_{0})(v) -\mathrm{D}_{e}^{2}\Phi(u_{0}|E_{0}-u_{0})(\overline{u}'_0),
$$
for all $v\in \HH^{1}(\Omega_{0})$. Hence one gets
\begin{multline}\label{signorini2casfaible}
    \dual{\overline{u}'_0}{v-\overline{u}'_0}_{\HH^{1}(\Omega_{0})} \geq\int_{\Omega_{0}}\mathrm{div}(f\boldsymbol{V})\left(v-\overline{u}'_0\right)-\int_{\Omega_{0}}\mathrm{div}(\boldsymbol{V})u_{0}\left(v-\overline{u}'_0\right)\\
    -\int_{\Omega_{0}}\left(-\nabla{\boldsymbol{V}}-\nabla{\boldsymbol{V}}^{\top}+\mathrm{div}(\boldsymbol{V})\mathrm{I}\right)\nabla{u_{0}}\cdot\nabla(v-\overline{u}'_0)\\
    +\int_{\Gamma_{0}}\left(\nabla{g}\cdot\boldsymbol{V}+ g\mathrm{div}_{\Gamma_0}(\boldsymbol{V})  \right)\frac{\partial_{\nn}u_{0}}{g}\left(v-\overline{u}'_0\right),
\end{multline}
for all $v\in\mathcal{K}_{u_{0},\frac{\partial_{\nn}\left(E_{0}-u_{0}\right)}{g}}$. Using the divergence formula (see Proposition~\ref{divergenceformula}) and the equality~$-\Delta u_0+u_0=f$ in $\LL^2(\Omega_0)$, we obtain that~$\overline{u}'_0$ is solution to~\eqref{matderivgeneral} and the uniqueness follows from the classical Stampacchia theorem~\cite{BREZ}.
\end{proof}

\begin{myRem}\normalfont\label{Remarquenotwice}
   Note that Equality~\eqref{hypoth1} in the second assumption of Theorem~\ref{Lagderiv} exactly corresponds to the inversion of the symbols $\mathrm{ME}\text{-}\mathrm{lim}$ and $\int_{\Gamma_0}$ in Equality~\eqref{Delta2}. In a general context, this is an open question. Nevertheless sufficient conditions can be derived and we refer to~\cite[Appendix~B]{4ABC} and~\cite[Appendix~A]{BCJDC} for examples.
\end{myRem}

\begin{myRem}\normalfont\label{remnonlinear}
Consider the framework of Theorem~\ref{Lagderiv} which is dependent of~$\boldsymbol{V} \in \mathcal{C}^{1,\infty}(\R^{d},\R^{d})$ and let us denote by~$\overline{u}'_0(\boldsymbol{V}) := \overline{u}'_0$. One can easily see that
$$
\overline{u}'_0( \alpha_1 \boldsymbol{V_{1}}+\alpha_2 \boldsymbol{V_{2}})= \alpha_1 \overline{u}'_0(\boldsymbol{V_{1}})+\alpha_2\overline{u}'_0(\boldsymbol{V_{2}}).
$$
for any $\boldsymbol{V_{1}}$, $\boldsymbol{V_{2}}\in\mathcal{C}^{1,\infty}(\R^{d},\R^{d})$ and for any nonnegative real numbers~$\alpha_1\geq0$, $\alpha_2 \geq 0$. However, this is not true for negative real numbers and justify why, in the present work, we call $\overline{u}'_0$ as material \textit{directional} derivative (instead of simply material derivative as usually in the literature). This nonlinearity is standard in shape optimization for variational inequalities (see, e.g.,~\cite{HINTERMULLERLAURAIN} or~\cite[Section 4]{SOKOZOL}).
\end{myRem}

The presentation of Theorem~\ref{Lagderiv} can be improved under additional regularity assumptions.

\begin{myCor}\label{LagderivSigno}
Consider the framework of Theorem~\ref{Lagderiv} with the additional assumptions that~$u_{0}\in\HH^{3}(\Omega_{0})$ and $\boldsymbol{V} \in \mathcal{C}^{2,\infty}(\R^{d},\R^{d}):=~\mathcal{C}^{2}(\R^{d},\R^{d})\cap\mathrm{W}^{2,\infty}(\R^{d},\R^{d})$.
Then $\overline{u}'_0\in\HH^{1}(\Omega_{0})$ is the unique weak solution to the scalar Signorini problem given by
\begin{equation}\label{caractu0DNT}
\arraycolsep=2pt
\left\{
\begin{array}{rcll}
-\Delta \overline{u}'_0+\overline{u}'_0 & = & -\Delta\left(\boldsymbol{V}\cdot\nabla{u_{0}}\right)+\boldsymbol{V}\cdot\nabla{u_{0}}  & \text{ in } \Omega_{0} , \\
\overline{u}'_0 & = & 0  & \text{ on } \Gamma^{u_{0},g}_{\mathrm{D}}, \\
\partial_{\nn} \overline{u}'_0 & = & h^m(\boldsymbol{V})  & \text{ on } \Gamma^{u_{0},g}_{\mathrm{N}}, \\
 \overline{u}'_0\leq0\text{, } \partial_{\nn}\overline{u}'_0\leq h^m(\boldsymbol{V}) \text{ and } \overline{u}'_0\left(\partial_{\nn}\overline{u}'_0-h^m(\boldsymbol{V})\right) & = & 0  & \text{ on } \Gamma^{u_{0},g}_{\mathrm{S-}}, \\
 \overline{u}'_0\geq0\text{, } \partial_{\nn}\overline{u}'_0\geq h^m(\boldsymbol{V}) \text{ and } \overline{u}'_0\left(\partial_{\nn}\overline{u}'_0- h^m(\boldsymbol{V})\right) & = & 0  & \text{ on } \Gamma^{u_{0},g}_{\mathrm{S+}},
\end{array}
\right.
\end{equation}
where $h^m(\boldsymbol{V}):= (\frac{\nabla{g}}{g}\cdot \boldsymbol{V}-\nabla{\boldsymbol{V}}\boldsymbol{\nn}\cdot\boldsymbol{\nn})\partial_{\nn}u_{0}+ (\nabla{\boldsymbol{V}}+\nabla{\boldsymbol{V}}^{\top})\nabla{u_{0}}\cdot\boldsymbol{\nn}\in\LL^{2}(\Gamma_{0})$.
\end{myCor}

\begin{proof}
Since $u_0 \in \HH^2(\Omega_0)$ and $\boldsymbol{V} \in \mathcal{C}^{2,\infty}(\R^{d},\R^{d})$, we deduce that~$\mathrm{div} ( ( -\nabla{\boldsymbol{V}} -\nabla{\boldsymbol{V}}^{\top}+\mathrm{div}(\boldsymbol{V})\mathrm{I} )\nabla{u_{0}} ) \in\LL^2(\Omega_0)$. Using the divergence formula (see Proposition~\ref{divergenceformula}) in Inequality~\eqref{matderivgeneral}, we get that
\begin{multline*}
        \dual{\overline{u}'_0}{v-\overline{u}'_0}_{\HH^{1}(\Omega_{0})} \geq\int_{\Omega_{0}}\boldsymbol{V}\cdot\nabla{u_0}\left(v-\overline{u}'_0\right)+\int_{\Omega_0}\Delta u_0 \boldsymbol{V}\cdot\nabla{\left(v-\overline{u}'_0\right)}
    \\+\int_{\Omega_{0}}\mathrm{div}\left(\left(-\nabla{\boldsymbol{V}}-\nabla{\boldsymbol{V}}^{\top}+\mathrm{div}(\boldsymbol{V})\mathrm{I}\right)\nabla{u_{0}}\right)\left(v-\overline{u}'_0\right)\\
    +\int_{\Gamma_{0}}\left(\boldsymbol{V}\cdot \boldsymbol{\nn} \left(f-u_0\right)+\left(\nabla{\boldsymbol{V}+\nabla{\boldsymbol{V}^{\top}}}\right)\nabla{u_0}\cdot \boldsymbol{\nn}+\left(\frac{\nabla{g}}{g}\cdot\boldsymbol{V}-\nabla{\boldsymbol{V}} \boldsymbol{\nn} \cdot \boldsymbol{\nn} \right)\partial_{\nn}u_{0}\right)\left(v-\overline{u}'_0\right),
\end{multline*}
for all $v\in\mathcal{K}_{u_{0},\frac{\partial_{\nn}\left(E_{0}-u_{0}\right)}{g}}$.
Moreover, since $\Delta u =u-f\in\HH^1(\Omega_0)$, it holds that $\mathrm{div}(\Delta u_0\boldsymbol{V})\in\LL^2(\Omega_0)$. Thus, using again the divergence formula, one deduces
\begin{multline}\label{lagderivu0H2}
    \dual{\overline{u}'_0}{v-\overline{u}'_0}_{\HH^{1}(\Omega_{0})}\geq\int_{\Omega_{0}}-\mathrm{div}\left(\left(\Delta u_0\right)\boldsymbol{V}-\mathrm{div}(\boldsymbol{V})\nabla{u_0}+(\nabla{\boldsymbol{V}}+\nabla{\boldsymbol{V}}^{\top})\nabla{u_0} \right)\left(v-\overline{u}'_0\right)\\+\int_{\Omega_{0}}\boldsymbol{V}\cdot\nabla{u_{0}}\left(v-\overline{u}'_0\right)+\int_{\Gamma_{0}}h^m(\boldsymbol{V})\left(v-\overline{u}'_0\right),
\end{multline}
for all $v\in\mathcal{K}_{u_{0},\frac{\partial_{\nn}\left(E_{0}-u_{0}\right)}{g}}$. Furthermore, one has $\Delta(\boldsymbol{V}\cdot\nabla{u_{0}})\in\LL^2(\Omega_0)$ from $u_0\in\HH^3(\Omega_0)$. Thus, using Proposition~\ref{divergencelaplacien}, it follows that 
\begin{equation*}
        \dual{\overline{u}'_0}{v-\overline{u}'_0}_{\HH^{1}(\Omega_{0})}\geq\int_{\Omega_{0}}-\Delta\left(\boldsymbol{V}\cdot\nabla{u_{0}}\right)\left(v-\overline{u}'_0\right)+\int_{\Omega_{0}}\boldsymbol{V}\cdot\nabla{u_{0}}\left(v-\overline{u}'_0\right)+\int_{\Gamma_{0}}h^m(\boldsymbol{V})\left(v-\overline{u}'_0\right),
\end{equation*}
for all $v\in\mathcal{K}_{u_{0},\frac{\partial_{\nn}\left(E_{0}-u_{0}\right)}{g}}$ which concludes the proof from Subsection~\ref{SectionSignorinicasscalairesansu}.
\end{proof}

\begin{myRem}\normalfont\label{regularityBrez}
If $\Gamma_0$ is sufficiently regular, then~$u_0 \in \HH^2(\Omega_0)$, and this is the best regularity result that can be obtained. We refer to~\cite[Chapter~1, Theorem I.10 p.43]{TheseBrez} and~\cite[Chapter~1, Remark I.26 p.47]{TheseBrez} for details. It does not mean that~$u_0 \notin \HH^3(\Omega_0)$ in general. It just means that, in this reference, there is a counterexample in which~$u_0 \notin \HH^3(\Omega_0)$ even if~$\Gamma_0$ is very smooth. Note that, from the proof of Corollary~\ref{LagderivSigno}, one can get, under the weaker assumption~$u_0 \in \HH^2(\Omega_0)$, that the material directional derivative $\overline{u}'_0$ is the solution to the variational inequality~\eqref{lagderivu0H2} which is, from Subsection~\ref{SectionSignorinicasscalairesansu}, the weak formulation of a Signorini problem with the source term given by~$-\mathrm{div} (\left(\Delta u_0\right)\boldsymbol{V}-\mathrm{div}(\boldsymbol{V})\nabla{u_0}+(\nabla{\boldsymbol{V}}+\nabla{\boldsymbol{V}}^{\top})\nabla{u_0} )+\boldsymbol{V}\cdot\nabla{u_{0}}\in\LL^2(\Omega_0)$.
\end{myRem}

Thanks to Corollary~\ref{LagderivSigno}, we are now in a position to characterize the shape directional derivative.

\begin{myCor}[Shape directional derivative]\label{shapederivaticesigno}
Consider the framework of Corollary~\ref{LagderivSigno} with the additional assumption that $\Gamma_0$ is of class $\mathcal{C}^3$. Then the shape directional derivative, defined by~$u'_{0}:=\overline{u}'_0-\nabla{u_{0}}\cdot\boldsymbol{V}\in\HH^{1}(\Omega_{0})$, is the unique weak solution to the scalar Signorini problem given by
\begin{equation*}\label{Eulesigno}
\arraycolsep=2pt
\left\{
\begin{array}{rcll}
-\Delta u_{0}'+u_{0}' & = & 0  & \text{ in } \Omega_{0} , \\
u_{0}' & = & -\boldsymbol{V}\cdot\nabla{u_{0}} & \text{ on } \Gamma^{u_{0},g}_{\mathrm{D}}, \\
\partial_{\nn} u_{0}' & = & h^s(\boldsymbol{V})  & \text{ on } \Gamma^{u_{0},g}_{\mathrm{N}}, \\
 u_{0}'\leq-\boldsymbol{V}\cdot\nabla{u_{0}}\text{, } \partial_{\nn}u_{0}'\leq h^s(\boldsymbol{V}) \text{ and } \left(u_{0}'+\boldsymbol{V}\cdot\nabla{u_{0}}\right)\left(\partial_{\nn}u_{0}'-h^s(\boldsymbol{V})\right) & = & 0  & \text{ on } \Gamma^{u_{0},g}_{\mathrm{S-}}, \\
 u_{0}'\geq-\boldsymbol{V}\cdot\nabla{u_{0}}\text{, } \partial_{\nn}u_{0}'\geq h^s(\boldsymbol{V}) \text{ and } \left(u_{0}'+\boldsymbol{V}\cdot\nabla{u_{0}}\right)\left(\partial_{\nn}u_{0}'- h^s(\boldsymbol{V})\right) & = & 0  & \text{ on } \Gamma^{u_{0},g}_{\mathrm{S+}},
\end{array}
\right.
\end{equation*}
where $h^s(\boldsymbol{V}):=\boldsymbol{V}\cdot \boldsymbol{\nn} (\partial_{\nn} (\partial_{\nn}u_0 )-\frac{\partial^2 u_0}{\partial \mathrm{n}^2} )+\nabla_{\Gamma_0} u_{0} \cdot\nabla_{\Gamma_0}(\boldsymbol{V}\cdot\boldsymbol{\mathrm{\nn}})-g\nabla{(\frac{\partial_{\nn}u_{0}}{g})}\cdot\boldsymbol{V}\in\LL^{2}(\Gamma_{0})$.
\end{myCor}

\begin{proof}
From the weak variational formulation of $\overline{u}'_0$ given in Corollary~\ref{LagderivSigno} and using the divergence formula (see Proposition~\ref{divergenceformula}), one can easily obtain that
$$
        \dual{u_{0}'}{v-\boldsymbol{V}\cdot\nabla{u_{0}}-u_{0}'}_{\HH^{1}(\Omega_{0})}\geq\int_{\Gamma_{0}}\left(h^m(\boldsymbol{V})-\nabla{(\boldsymbol{V}\cdot  \nabla u_{0} } ) \cdot\boldsymbol{\nn}\right)\left(v-\boldsymbol{V}\cdot\nabla{u_{0}}-u_{0}'\right),
$$
for all $v\in\mathcal{K}_{u_{0},\frac{\partial_{\nn}\left(E_{0}-u_{0}\right)}{g}}$ (see notation introduced in Theorem~\ref{Lagderiv}), which can be rewritten as
$$ 
\dual{u_{0}'}{w-u_{0}'}_{\HH^{1}(\Omega_{0})}\geq\int_{\Gamma_{0}}\left(h^m(\boldsymbol{V})-\nabla{(\boldsymbol{V}\cdot\nabla{u_{0}})}\cdot\boldsymbol{\nn}\right)\left(w-u_{0}'\right),
$$
for all $w\in\mathcal{K}_{u_{0},\frac{\partial_{\nn}\left(E_{0}-u_{0}\right)}{g}}-\boldsymbol{V}\cdot\nabla{u_{0}}$. Since $\Gamma_0$ is of class $\mathcal{C}^{3}$ and $u_0\in\HH^3(\Omega_0)$, the normal derivative of~$u_{0}$ can be extended into a function defined in $\Omega_{0}$ such that~$\partial_{\nn}u_{0}\in\HH^{2}(\Omega_{0})$. Thus, it holds that~$v\partial_{\nn}u_{0}\in\mathrm{W}^{2,1}(\Omega_{0})$ for all~$v\in\mathcal{C}^{\infty}(\overline{\Omega_{0}})$, and one can use Propositions~\ref{divergencelaplacien} and~\ref{intbord} to obtain that 
\begin{multline*}
\int_{\Gamma_{0}}\left(h^m(\boldsymbol{V})-\nabla{(\boldsymbol{V}\cdot\nabla{u_{0}})}\cdot\boldsymbol{\nn}\right)v\\=\int_{\Gamma_{0}}\boldsymbol{V}\cdot\boldsymbol{\nn}\left(-\nabla{u_{0}}\cdot\nabla{v}-u_{0}v+fv+Hv\partial_{\nn}u_{0}+\partial_{\nn}\left(v\partial_{\nn}u_{0}\right)\right)
-\int_{\Gamma_{0}}gv\nabla{\left(\frac{\partial_{\nn}u_{0}}{g}\right)}\cdot\boldsymbol{V},
\end{multline*}
for all $v\in\mathcal{C}^{\infty}(\overline{\Omega_{0}})$.
Then, by using Proposition~\ref{beltrami}, one deduces that
\begin{multline*}
    \int_{\Gamma_{0}}\left(h^m(\boldsymbol{V})-\nabla{(\boldsymbol{V}\cdot\nabla{u_{0}})}\cdot\boldsymbol{\nn}\right)v\\=\int_{\Gamma_{0}}\left(\boldsymbol{V}\cdot \boldsymbol{\nn} \left(\partial_{\nn}\left(\partial_{\nn}u_0\right)-\frac{\partial^2 u_0}{\partial \mathrm{n}^2}\right)+\nabla_{\Gamma_{0}}{u_{0}}\cdot\nabla_{\Gamma_{0}}{\left(\boldsymbol{V}\cdot\boldsymbol{\nn}\right)}-g\nabla{\left(\frac{\partial_{\nn}u_{0}}{g}\right)}\cdot\boldsymbol{V}\right)v,
\end{multline*}
for all $v\in\mathcal{C}^{\infty}(\overline{\Omega_{0}})$, and also for all $v\in\HH^{1}(\Omega_{0})$ by density. Thus it follows that
\begin{multline*}
    \dual{u_{0}'}{w-u_{0}'}_{\HH^{1}(\Omega_{0})}\\\geq\int_{\Gamma_{0}}\left(\boldsymbol{V}\cdot\boldsymbol{\nn}\left(\partial_{\nn}\left(\partial_{\nn}u_0\right)-\frac{\partial^2 u_0}{\partial \mathrm{n}^2}\right)+\nabla_{\Gamma_{0}}{u_{0}}\cdot\nabla_{\Gamma_{0}}{\left(\boldsymbol{V}\cdot\boldsymbol{\nn}\right)}-g\nabla{\left(\frac{\partial_{\nn}u_{0}}{g}\right)}\cdot\boldsymbol{V}\right)\left(w-u_{0}'\right),
\end{multline*}
for all $w\in\mathcal{K}_{u_{0},\frac{\partial_{\nn}\left(E_{0}-u_{0}\right)}{g}}-\boldsymbol{V}\cdot\nabla{u_{0}}$,
which concludes the proof from Subsection~\ref{SectionSignorinicasscalairesansu}.
\end{proof}

\subsection{Shape gradient of the Tresca energy functional}\label{energyfunctional}
Thanks to the characterization of the material directional derivative obtained in Theorem~\ref{Lagderiv}, we are now in a position to prove the main result of the present paper.

\begin{myTheorem}\label{shapederivofJ}
Consider the framework of Theorem~\ref{Lagderiv}. Then the Tresca energy functional~$\mathcal{J}$ admits a shape gradient at $\Omega_{0}$ in any direction~$\boldsymbol{V} \in\mathcal{C}^{1,\infty}(\R^{d},\R^{d})$ given by
\begin{multline}\label{gradientdeforme1}
    \mathcal{J}'(\Omega_{0})(\boldsymbol{V})=\frac{1}{2}\int_{\Omega_0}\mathrm{div}(\boldsymbol{V})\left\|\nabla{u_0}\right\|^2-\int_{\Omega_0}\nabla{u_0}\cdot\left(\nabla{\boldsymbol{V}}\nabla{u_0}+\Delta u_0\boldsymbol{V}\right)\\
    +\int_{\Gamma_0}\left(\boldsymbol{V}\cdot\boldsymbol{\nn}\left(\frac{\left|u_0\right|^2}{2}-fu_0\right)-\left(\frac{\nabla{g}}{g}\cdot\boldsymbol{V}+\mathrm{div}_{\Gamma_0}(\boldsymbol{V})\right)u_0\partial_{\nn}u_{0}\right).
\end{multline}
\end{myTheorem}

\begin{proof}
By following the usual strategy developed in the shape optimization literature (see, e.g.,~\cite{ALL,HENROT}) to compute the shape gradient of $\mathcal{J}$ at $\Omega_{0}$ in a direction $\boldsymbol{V}\in\mathcal{C}^{1,\infty}(\R^{d},\R^{d})$, one gets
\begin{equation*}
\mathcal{J}'(\Omega_{0})(\boldsymbol{V})=-\frac{1}{2}\int_{\Omega_{0}}\left(\left\|\nabla{u_{0}}\right\|^{2}+\left|u_{0}\right|^{2}\right)\mathrm{div}(\boldsymbol{V})+\int_{\Omega_{0}}\nabla{u_{0}}\cdot\nabla{\boldsymbol{V}}\nabla{u_{0}}-\dual{\overline{u}'_0}{u_{0}}_{\HH^{1}(\Omega_{0})}.
\end{equation*}
On the other hand, since $\overline{u}'_0\pm u_{0}\in\mathcal{K}_{u_{0},\frac{\partial_{\nn}\left(E_{0}-u_{0}\right)}{g}}$ (see notation introduced in Theorem~\ref{Lagderiv}), we deduce from the weak variational formulation of $\overline{u}'_0$ that
\begin{multline*}
    \dual{\overline{u}'_0}{u_0}_{\HH^{1}(\Omega_{0})} =\int_{\Omega_{0}}u_0\boldsymbol{V}\cdot\nabla{u_0}
    \\-\int_{\Omega_{0}}\left(\left(-\nabla{\boldsymbol{V}}-\nabla{\boldsymbol{V}}^{\top}+\mathrm{div}(\boldsymbol{V})\mathrm{I}\right)\nabla{u_{0}}-\Delta u_0 \boldsymbol{V}\right)\cdot\nabla{u_0}\\
    +\int_{\Gamma_{0}}\left(\boldsymbol{V}\cdot\boldsymbol{\nn}\left(fu_0-\left|u_0\right|^2\right)+\left(\frac{\nabla{g}}{g}\cdot\boldsymbol{V}+\mathrm{div}_{\Gamma_0}(\boldsymbol{V})\right)u_0\partial_{\nn}u_{0}\right).
\end{multline*}
The proof is complete thanks to the divergence formula (see Proposition~\ref{divergenceformula}).
\end{proof}

As we did in Corollary~\ref{LagderivSigno} for the material directional derivative, the presentation of Theorem~\ref{shapederivofJ} can be improved under additional assumptions.

\begin{myCor}\label{shapederivofJ3}
Consider the framework of Theorem~\ref{shapederivofJ} with the additional assumptions that $d\in\left\{1,2,3,4,5\right\}$, $\Gamma_0$ is of class $\mathcal{C}^{3}$ and $u_0\in\HH^3(\Omega_0)$. Then the shape gradient of the Tresca energy functional $\mathcal{J}$ at $\Omega_{0}$ in any direction~$\boldsymbol{V} \in \mathcal{C}^{1,\infty}(\R^{d},\R^{d})$ is given by
$$ \mathcal{J}'(\Omega_{0})(\boldsymbol{V})=\int_{\Gamma_{0}} \boldsymbol{V}\cdot\boldsymbol{\nn}\left(\frac{\left\|\nabla{u_{0}}\right\|^{2}+\left|u_{0}\right|^{2}}{2}-fu_{0}+Hg\left| u_{0} \right| -\partial_{\nn}\left(u_{0}\partial_{\nn}u_{0}\right)+gu_{0}\nabla{\left(\frac{\partial_{\nn}u_{0}}{g}\right)}\cdot\boldsymbol{\nn}\right),
$$
where $H$ is the mean curvature of $\Gamma_{0}$.
\end{myCor}

\begin{proof}
Let~$\boldsymbol{V}\in~\mathcal{C}^{1,\infty}(\R^{d},\R^{d})$. Since $u_0\in\HH^2(\Omega_0)\subset\HH^3(\Omega_0)$, it holds that
$$
\int_{\Omega_0}\mathrm{div}(\boldsymbol{V})\left\|\nabla{u_0}\right\|^2=-\int_{\Omega_0}\boldsymbol{V}\cdot\nabla{\left(\left\|\nabla{u_0}\right\|^2\right)}+\int_{\Gamma_0}\boldsymbol{V}\cdot \boldsymbol{\nn} \left\|\nabla{u_0}\right\|^2,
$$
and
$$
\int_{\Omega_0}\Delta u_0\boldsymbol{V}\cdot\nabla{u_0}=-\int_{\Omega_0}\nabla{u_0}\cdot\nabla{\left(\boldsymbol{V}\cdot\nabla{u_0}\right)}+\int_{\Gamma_0}\partial_{\nn}u_0\boldsymbol{V}\cdot\nabla{u_0}.
$$
One deduces from~\eqref{gradientdeforme1} that
\begin{multline}\label{gradientdeforme2}
    \mathcal{J}'(\Omega_{0})(\boldsymbol{V})=\int_{\Gamma_{0}}\boldsymbol{V}\cdot\boldsymbol{\nn}\left(\frac{\left\|\nabla{u_{0}}\right\|^{2}+\left|u_{0}\right|^{2}}{2}-fu_{0}\right)\\-\int_{\Gamma_0}\left(\partial_{\nn}u_0\boldsymbol{V}\cdot\nabla{u_0}+\left(\frac{\nabla{g}}{g}\cdot\boldsymbol{V}+\mathrm{div}_{\Gamma_0}(\boldsymbol{V})\right)u_0\partial_{\nn}u_{0}\right).
\end{multline}
Moreover, since $\Gamma_0$ is of class $\mathcal{C}^{3}$ and $u_0\in\HH^3(\Omega_0)$, the normal derivative of $u_{0}$ can be extended into a function defined in $\Omega_{0}$ such that~$\partial_{\nn}u_{0}\in\HH^{2}(\Omega_{0})$. Therefore, using Proposition~\ref{intbord} with~$v=u_{0}\partial_{\nn}u_{0}\in\mathrm{W}^{2,1}(\Omega_{0})$, one gets
$$
\mathcal{J}'(\Omega_{0})(\boldsymbol{V})=\int_{\Gamma_{0}}\boldsymbol{V}\cdot\boldsymbol{\nn}\left(\frac{\left\|\nabla{u_{0}}\right\|^{2}+\left|u_{0}\right|^{2}}{2}-fu_{0}-Hu_{0}\partial_{\nn}u_{0}-\partial_{\nn}\left(u_{0}\partial_{\nn}u_{0}\right)\right)
+\int_{\Gamma_{0}}gu_{0}\nabla{\left(\frac{\partial_{\nn}u_{0}}{g}\right)}\cdot\boldsymbol{V}.
$$
From the scalar Tresca friction law, one has $Hu_{0}\partial_{\nn}u_{0}=-Hg|u_{0}|$ \textit{a.e.} on $\Gamma_{0}$.
Now let us focus on the last term. Since  $u_{0}=0$ on $\Gamma^{u_{0},g}_{\mathrm{D}}\cup\Gamma^{u_{0},g}_{\mathrm{S-}}\cup\Gamma^{u_{0},g}_{\mathrm{S+}}$, we have
$$
\int_{\Gamma_0}gu_{0}\nabla{\left(\frac{\partial_{\nn}u_{0}}{g}\right)}\cdot\boldsymbol{V}=\int_{\Gamma^{u_{0},g}_{\mathrm{N}}}gu_{0}\nabla{\left(\frac{\partial_{\nn}u_{0}}{g}\right)}\cdot\boldsymbol{V}.
$$
Let us introduce two disjoint subsets of $\Gamma_0$ given by
$$
\Gamma^{u_{0},g}_{\mathrm{N+}}:=\left\{s\in\Gamma_{0} \mid u_{0}(s)>0\right\}
\qquad \mbox{ and } \qquad 
\Gamma^{u_{0},g}_{\mathrm{N}-}:=\left\{s\in\Gamma_{0} \mid  u_{0}(s)<0\right\}.
$$
Hence it follows that $\Gamma^{u_{0},g}_{\mathrm{N}}=\Gamma^{u_{0},g}_{\mathrm{N+}}\cup\Gamma^{u_{0},g}_{\mathrm{N-}}$, with $\partial_{\nn}u_{0}=-g$ \textit{a.e.} on $\Gamma^{u_{0},g}_{\mathrm{N+}}$, and~$\partial_{\nn}u_{0}=g$ \textit{a.e.} on~$\Gamma^{u_{0},g}_{\mathrm{N-}}$. Moreover, since $u_{0}\in\HH^{3}(\Omega)$ and $d\in\left\{1,2,3,4,5\right\}$, we get from Sobolev embeddings (see, e.g.,~\cite[Chapter 4, p.79]{ADAMS}) that $u_{0}$ is continuous over $\Gamma_{0}$, thus $\Gamma^{u_{0},g}_{\mathrm{N+}}$ and $\Gamma^{u_{0},g}_{\mathrm{N-}}$ are open subsets of~$\Gamma_{0}$. Hence $\nabla_{\Gamma_0}(\frac{\partial_{\nn}u_0}{g})=0$ \textit{a.e.} on $\Gamma^{u_{0},g}_{\mathrm{N+}}\cup\Gamma^{u_{0},g}_{\mathrm{N-}}$, and one deduces that 
$$
\int_{\Gamma^{u_{0},g}_{\mathrm{N}}}gu_{0}\nabla{\left(\frac{\partial_{\nn}u_{0}}{g}\right)}\cdot\boldsymbol{V}=\int_{\Gamma^{u_{0},g}_{\mathrm{N}}}\boldsymbol{V}\cdot\boldsymbol{\nn}\left(gu_{0}\nabla{\left(\frac{\partial_{\nn}u_{0}}{g}\right)}\cdot\boldsymbol{\nn}\right),
$$
which concludes the proof.
\end{proof}

\begin{myRem}\normalfont
Under the weaker condition~$u_0\in\HH^2(\Omega_0)$ (satisfied if $\Gamma_0$ is sufficiently regular, see Remark~\ref{regularityBrez}), one can follow the proof of Corollary~\ref{shapederivofJ3} and obtain that the shape gradient of $\mathcal{J}$ is given by Equality~\eqref{gradientdeforme2}.
\end{myRem}

\begin{myRem}\label{remarkadjoint}\normalfont
Consider the framework of Theorem~\ref{shapederivofJ}. We have seen in Remark~\ref{remnonlinear} that the expression of the material directional derivative~$\overline{u}'_0$ is not linear with respect to~$\boldsymbol{V}$. However one can observe that the scalar product~$\dual{\overline{u}'_0}{u_{0}}_{\HH^{1}(\Omega_{0})}$, that appears in the proof of Theorem~\ref{shapederivofJ}, is. This leads to an expression of the shape gradient~$\mathcal{J}'(\Omega_{0})(\boldsymbol{V})$ in Theorem~\ref{shapederivofJ} that is linear with respect to~$\boldsymbol{V}$. Hence we deduce that the Tresca energy functional $\mathcal{J}$ is shape
differentiable at $\Omega_{0}$. Note that, in the context of cracks and variational
inequalities involving unilateral conditions, it can already be observed that the shape gradient of the energy functional is linear with respect to $\boldsymbol{V}$ (see, e.g.,~\cite[Theorem 2.22 or Theorem 4.20]{Gilles} and references therein). Furthermore note that the shape gradient~$\mathcal{J}'(\Omega_{0})(\boldsymbol{V})$ depends only on $u_0$ (and not on $u’_0$) and therefore does not require the introduction of an appropriate adjoint problem to be computed explicitly. The linear explicit expression of~$\mathcal{J}'(\Omega_{0})(\boldsymbol{V})$ with respect to the direction $\boldsymbol{V}$ will allow us in the next Section~\ref{numericalsim} to exhibit a descent direction for numerical simulations in order to solve the shape optimization problem~\eqref{shapeOptim} on a two-dimensional example. It is worth noting that all previous comments are specific to the Tresca energy functional~$\mathcal{J}$. Other cost functionals, such as the least-square functional, can pose challenges to correctly define an adjoint problem due to nonlinearities in shape gradients. Note that these difficulties do not appear in the literature when using regularization procedures (see, e.g.,~\cite{HINTERMULLERLAURAIN}). Our approach, which is solely based on convex and variational analysis, does not address this challenge yet, and we believe it is an interesting area for future research.
\end{myRem}

\begin{myRem}\label{energyfunctionalNeu}\normalfont
Let us recall that the standard Neumann energy functional is
\begin{equation*}
    \mathcal{J}_\mathrm{N}(\Omega) := \frac{1}{2}\int_{\Omega}\left( \left\| \nabla{w_{\mathrm{N},\Omega}} \right\|^{2}+|w_{\mathrm{N},\Omega}|^{2}\right)+\int_{\Gamma}gw_{\mathrm{N},\Omega}-\int_{\Omega}fw_{\mathrm{N},\Omega},
\end{equation*}
for all $\Omega\in\mathcal{U}$, where $w_{\mathrm{N},\Omega}\in\HH^1(\Omega)$ is the unique solution to the standard Neumann problem
\begin{equation}\label{Neumannproblem222}\tag{SNP\ensuremath{{}_\Omega}}
    \arraycolsep=2pt
\left\{
\begin{array}{rcll}
-\Delta w_{\mathrm{N},\Omega}+w_{\mathrm{N},\Omega} &=&f    & \text{ in } \Omega, \\
\partial_{\nn}w_{\mathrm{N},\Omega}&=&-g  & \text{ on } \Gamma.\\
\end{array}
\right.
\end{equation}
One can prove (see, e.g.,~\cite{ALL,HENROT}) that the shape gradient of the Neumann energy functional $\mathcal{J}_\mathrm{N}$ at~$\Omega_{0}\in\mathcal{U}$ in any direction $\boldsymbol{V}\in\mathcal{C}^{1,\infty}(\R^{d},\R^{d})$ is given by 
$$
\mathcal{J}_\mathrm{N}'(\Omega_{0})(\boldsymbol{V})=\int_{\Gamma_{0}} \boldsymbol{V}\cdot \boldsymbol{\nn} \left(\frac{\left\|\nabla{w_{\mathrm{N},\Omega_{0}}}\right\|^{2}+\left|w_{\mathrm{N},\Omega_{0}}\right|^{2}}{2}-fw_{\mathrm{N},\Omega_{0}}+Hgw_{\mathrm{N},\Omega_{0}} +\partial_{\nn}\left(gw_{\mathrm{N},\Omega_{0}}\right)\right).
$$
Thus the shape gradient of the Tresca energy functional $\mathcal{J}$ obtained in Corollary~\ref{shapederivofJ3} is close to the one of $\mathcal{J}_\mathrm{N}$ with the additional term
$$
\int_{\Gamma_{0}}gu_{0}\nabla{\left(\frac{\partial_{\nn}u_{0}}{g}\right)}\cdot\boldsymbol{V}.
$$
Note that, if $\partial_{\nn}u_{0}=-g$ \textit{a.e.} on $\Gamma_{0}$, then they coincide.
\end{myRem}

\begin{myRem}\label{energyfunctionalDP}\normalfont
Let us recall that the standard Dirichlet energy functional is
\begin{equation*}
    \mathcal{J}_\mathrm{D}(\Omega) := \frac{1}{2}\int_{\Omega}\left( \left\| \nabla{w_{\mathrm{D},\Omega}} \right\|^{2}+|w_{\mathrm{D},\Omega}|^{2}\right)-\int_{\Omega}fw_{\mathrm{D},\Omega},
\end{equation*}
for all $\Omega\in\mathcal{U}$, where $w_{\mathrm{D},\Omega}\in\HH^1(\Omega)$ is the unique solution to the Dirichlet problem
\begin{equation}\label{Dirichletproblem22}\tag{DP\ensuremath{{}_\Omega}}
    \arraycolsep=2pt
\left\{
\begin{array}{rcll}
-\Delta w_{\mathrm{D},\Omega}+w_{\mathrm{D},\Omega} &=&f    & \text{ in } \Omega, \\
w_{\mathrm{D},\Omega}&=&0  & \text{ on } \Gamma.\\
\end{array}
\right.
\end{equation}
One can prove (see, e.g.,~\cite{ALL,HENROT}) that the shape gradient of $\mathcal{J}_\mathrm{D}$ at $\Omega_{0}\in\mathcal{U}$ in any direction~$\boldsymbol{V}\in\mathcal{C}^{1,\infty}(\R^{d},\R^{d})$ is given by
$$
\mathcal{J}_\mathrm{D}'(\Omega_{0})(\boldsymbol{V})=-\int_{\Gamma_{0}} \boldsymbol{V}\cdot \boldsymbol{\nn}\left(\frac{\left\|\nabla{w_{\mathrm{D},\Omega_0}}\right\|^{2}+\left|w_{\mathrm{D},\Omega_0}\right|^{2}}{2}\right).
$$
Note that, if $u_0=0$ \textit{a.e.} on $\Gamma_0$, then $\nabla_{\Gamma_0}u_0=0$ \textit{a.e.} on $\Gamma_0$, thus $(\partial_{\nn}u_0)^2=||\nabla{u_0}||^2$ \textit{a.e.} on $\Gamma_0$ and thus the shape gradient of $\mathcal{J}$ obtained in Corollary~\ref{shapederivofJ3} coincides with the one of $\mathcal{J}_\mathrm{D}$.
\end{myRem}

\section{Numerical simulations}\label{numericalsim}
In this section we numerically solve an example of the shape optimization problem~\eqref{shapeOptim} in the two-dimensional case $d=2$, by making use of our theoretical results obtained in Section~\ref{mainresult}. The numerical simulations have been performed using Freefem++ software~\cite{HECHT} with P1-finite elements and standard affine mesh. We could use the expression of the shape gradient of~$\mathcal{J}$ obtained in Theorem~\ref{shapederivofJ} but, for the purpose of simplifying the computations, we chose to use the expression provided in Corollary~\ref{shapederivofJ3} under additional assumptions (such as~$u_0 \in \mathrm{H}^3(\Omega_0)$ that we assumed to be true at each iteration). The~$\mathcal{C}^3$ regularity of the shapes required in Corollary~\ref{shapederivofJ3} is not satisfied since we use a classical affine mesh and thus the discretized domains have boundaries that are only Lipschitz. Nevertheless it could be possible to impose more regularity by using curved mesh for example. However the use of such numerical techniques falls outside the scope of this paper in which the numerical simulations are intended to illustrate our theoretical results.

\subsection{Numerical methodology}\label{presentexample}
Consider an initial shape~$\Omega_0 \in \mathcal{U}$ (see the beginning of Section~\ref{mainresult} for the definition of~$\mathcal{U}$). Note that Corollary~\ref{shapederivofJ3} allows to exhibit a descent direction~$\boldsymbol{V_{0}}$ of the Tresca energy functional~$\mathcal{J}$ at~$\Omega_0$ as the unique solution to the Neumann problem
\begin{equation*}
\arraycolsep=2pt
\left\{
\begin{array}{rcll}
-\Delta \boldsymbol{V_{0}}+\boldsymbol{V_{0}} & = & 0  & \text{ in } \Omega_{0} , \\
\nabla{\boldsymbol{V_{0}}}\boldsymbol{\nn} & = &- \left(\frac{\left\|\nabla{u_{0}}\right\|^{2}+\left|u_{0}\right|^{2}}{2}-fu_{0}+Hg\left| u_{0} \right| -\partial_{\nn}\left(u_{0}\partial_{\nn}u_{0}\right)+gu_{0}\nabla{\left(\frac{\partial_{\nn}u_{0}}{g}\right)}\cdot\boldsymbol{\nn}\right)\boldsymbol{\nn} & \text{ on } \Gamma_{0},
\end{array}
\right.
\end{equation*}
since it satisfies~$\mathcal{J}'(\Omega_{0})(\boldsymbol{V_0})=-\left\|\boldsymbol{V_0}\right\|^{2}_{\HH^{1}(\Omega_0)^{d}}\leq0
$. 

In order to numerically solve the shape optimization problem~\eqref{shapeOptim} on a given example, we also have to deal with the volume constraint~$\vert \Omega \vert = \lambda>0$. To this aim, the Uzawa algorithm (see, e.g.,~\cite[Chapter 3 p.64]{ALL}) is used. In a nutshell it consists in augmenting the Tresca energy functional $\mathcal{J}$ by adding an initial Lagrange multiplier~$p_0 \in \R$ multiplied by the standard volume functional minus~$\lambda$. From~\cite[Chapter~6, Section 6.5]{ALL}, we know that the shape gradient of the volume functional at $\Omega_0$ is given by
$$\boldsymbol{V}\in\mathcal{C}^{1,\infty}(\R^{d},\R^{d}) \mapsto \int_{\Gamma_0}\boldsymbol{V}\cdot\boldsymbol{\nn}\in\mathbb{R},
$$
and thus one can easily obtain a descent direction~$\boldsymbol{V_0}(p_0)$ of the \textit{augmented} Tresca energy functional at~$\Omega_0$ by adding~$p_0\boldsymbol{\nn}$ in the Neumann boundary condition of~$\boldsymbol{V_0}$. This descent direction leads to a new shape $\Omega_{1}:=(\boldsymbol{\mathrm{id}}+\tau\boldsymbol{V_{0}}(p_0))(\Omega_{0})$, where~$\tau>0$ is a fixed parameter. Finally the Lagrange multiplier is updated as follows
$$
p_{1}:=p_{0}+\mu\left( \vert \Omega_{1} \vert-\lambda\right),
$$
where $\mu>0$ is a fixed parameter, and the algorithm restarts with~$\Omega_1$ and~$p_{1}$, and so on.

Let us mention that the scalar Tresca friction problem is numerically solved using an adaptation of iterative switching algorithms (see~\cite{11AIT}). This algorithm operates by checking at each iteration if the Tresca boundary conditions are satisfied and, if they are not, by imposing them and restarting the computation (see~\cite[Appendix C p.25]{4ABC} for detailed explanations). We also precise that, for all~$i\in~\mathbb{N}^{*}$, the difference between the Tresca energy functional $\mathcal{J}$ at the iteration $20\times i$ and at the iteration~$20\times (i-1)$ is computed. The smallness of this difference is used as a stopping criterion for the algorithm. Finally the curvature term~$H$ is numerically computed by extending the normal~$\boldsymbol{\mathrm{n}}$ into a function~$\tilde{\boldsymbol{\mathrm{n}}}$ which is defined on the whole domain~$\Omega_0$. Then the curvature is given by~$H=\mathrm{div}(\tilde{\boldsymbol{\mathrm{n}}})-\nabla(\tilde{\boldsymbol{\mathrm{n}}}) \boldsymbol{\mathrm{n}} \cdot \boldsymbol{\mathrm{n}}$ (see, e.g.,~\cite[Proposition 5.4.8 p.194]{HENROT}).

\subsection{Two-dimensional example and numerical results}
In this subsection, take~$d=2$ and $f\in\HH^{1}(\R^{2})$ given by
$$
\fonction{f}{\R^{2}}{\R}{(x,y)}{\displaystyle f(x,y)= \frac{5-x^{2}-y^{2}+xy}{4}\eta(x,y),}
$$
and, for a given parameter $\beta > 0$, let $g_\beta\in\HH^{2}(\R^{2})$ be given by
$$
\fonction{g_\beta}{\R^{2}}{\R}{(x,y)}{\displaystyle g(x,y)=\beta\left(1+\frac{(\sin{x})^{2}}{0.8}\right)\eta(x,y),}
$$
where $\eta\in\mathcal{C}^{\infty}_0(\R^2)$ is a cut-off function chosen appropriately so that $f$ and $g$ satisfy the assumptions of the present paper. The volume constraint considered is $\lambda=\pi$ and the initial shape $\Omega_{0}\subset\R^{2}$ is an ellipse centered at $(0,0)\in\R^2$, with semi-major axis $a=1.3$ and semi-minor axis $b=1/a$. 

In what follows, we present the numerical results obtained for this two-dimensional example using the methodology described in Subsection~\ref{presentexample}, and for different values of $\beta$:

\smallskip

\begin{itemize}
\item Figure~\ref{optishapeTrescaDirichlet} shows on the left the shape which solves Problem~\eqref{shapeOptim} for $\beta=0.49$, and on the right the one when the Tresca problem and its energy functional are replaced by Dirichlet ones (see Remark~\ref{energyfunctionalDP}). We observe that both shapes are very close. Indeed, with~$\beta \geq 0.49$, one can check numerically that the solution $w_{\mathrm{D},\Omega}$ to the Dirichlet problem~\eqref{Dirichletproblem22} satisfies~$|\partial_{\nn}w_{\mathrm{D},\Omega}|<g_{\beta}$ on~$\Gamma$, and thus is also the solution to the scalar Tresca friction problem~\eqref{Trescaproblem222}. One deduces from Remark~\ref{energyfunctionalDP} that the shape gradient of $\mathcal{J}$ and the one of $\mathcal{J}_\mathrm{D}$ coincide. Therefore, since the shape minimizing the Dirichlet energy functional~$\mathcal{J}_{\mathrm{D}}$ under the volume constraint $\lambda=\pi$ is a critical shape of the \textit{augmented} Dirichlet energy functional, it is also a critical shape of the \textit{augmented} Tresca energy functional.
\begin{figure}[h!]
    \centering
    \includegraphics[scale=0.45, trim = 4.2cm 0.1cm 4.4cm 0.6cm, clip]{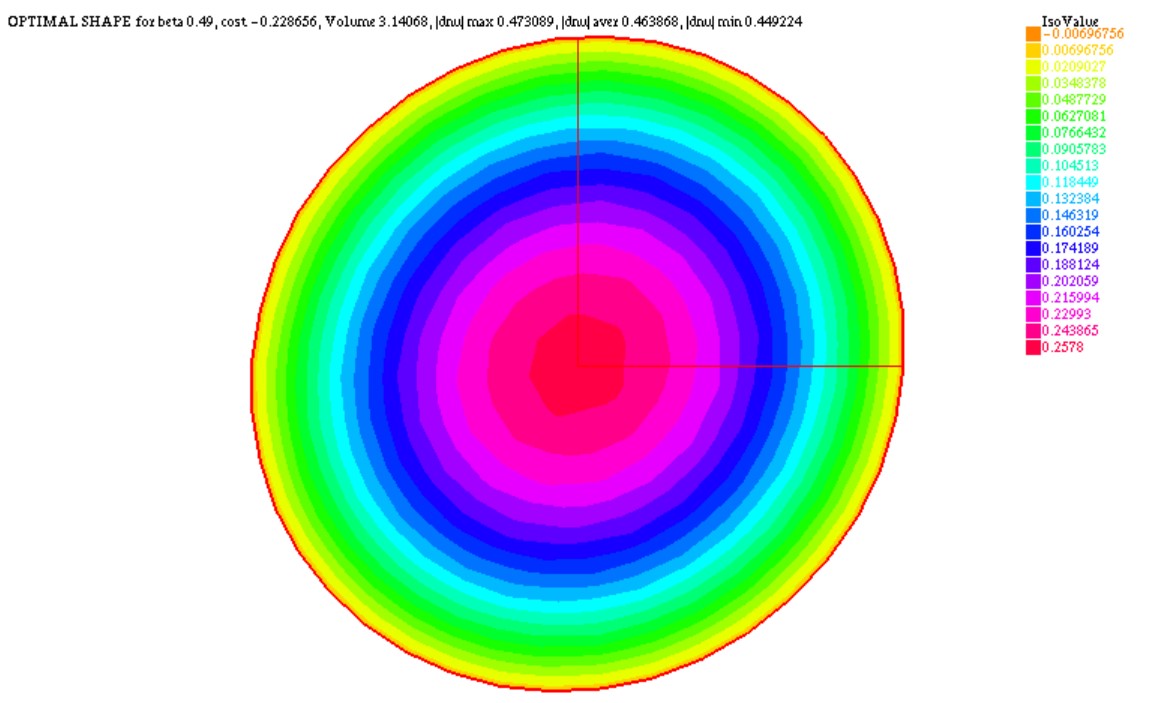} \hspace{2cm} 
    \includegraphics[scale=0.45, trim = 0.2cm 0.3cm 4.cm 0.7cm, clip]{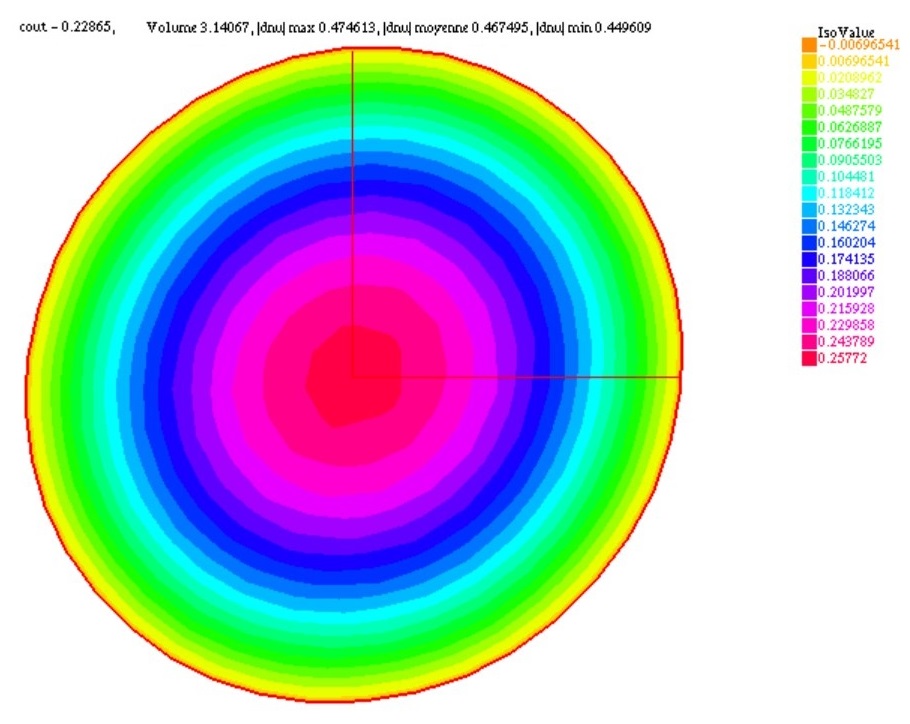}
    \caption{Shapes minimizing $\mathcal{J}$ (left) and $\mathcal{J}_{\mathrm{D}}$ (right), under the volume constraint $\lambda=\pi$, and with $\beta=0.49$.}
    \label{optishapeTrescaDirichlet}
\end{figure}

\smallskip

\item Figure~\ref{optishapeTresca} shows the shapes which solve Problem~\eqref{shapeOptim} for $\beta=0.46, 0.43, 0.37, 0.31$. The shapes are different from the one obtained on the left of Figure~\ref{optishapeTrescaDirichlet}. In that context, note that the normal derivative of the solution $u$ to the scalar Tresca friction problem~\eqref{Trescaproblem222} reaches the friction threshold $g_\beta$ on some parts of the boundary.
\begin{figure}[h!]
    \centering
    \includegraphics[scale=0.45, trim = 4.3cm -0.4cm 4.2cm 0.6cm, clip]{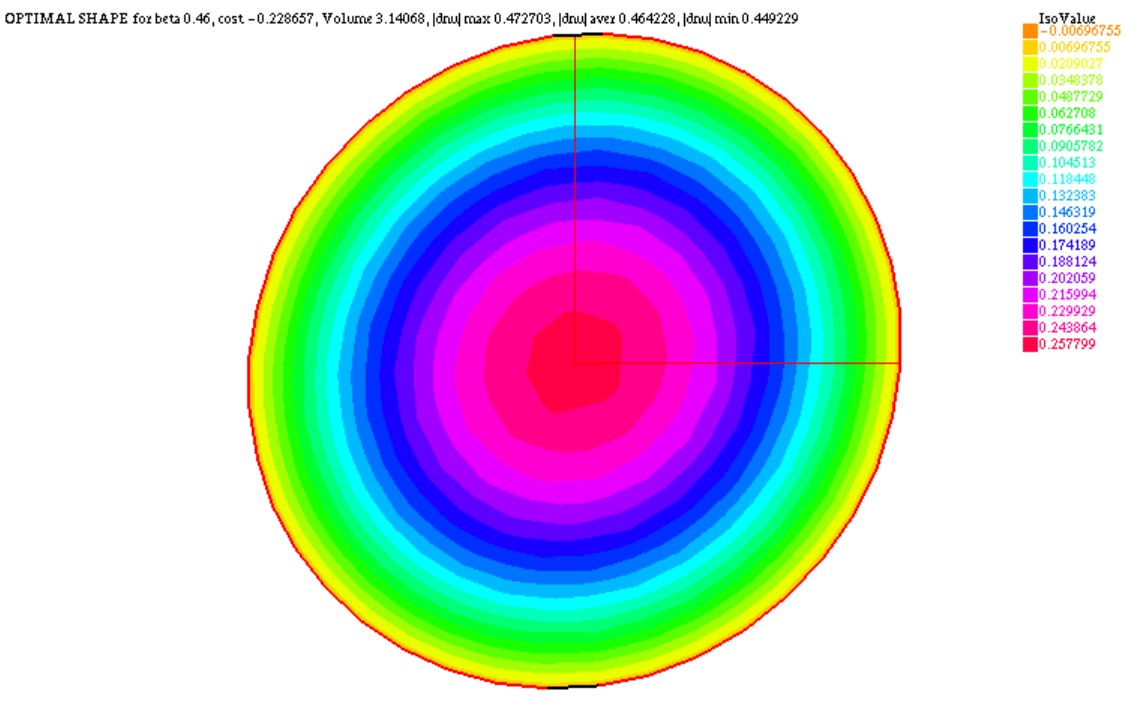} \hspace{2cm} 
    \includegraphics[scale=0.45, trim = 4.1cm -0.4cm 4.2cm 0.6cm, clip]{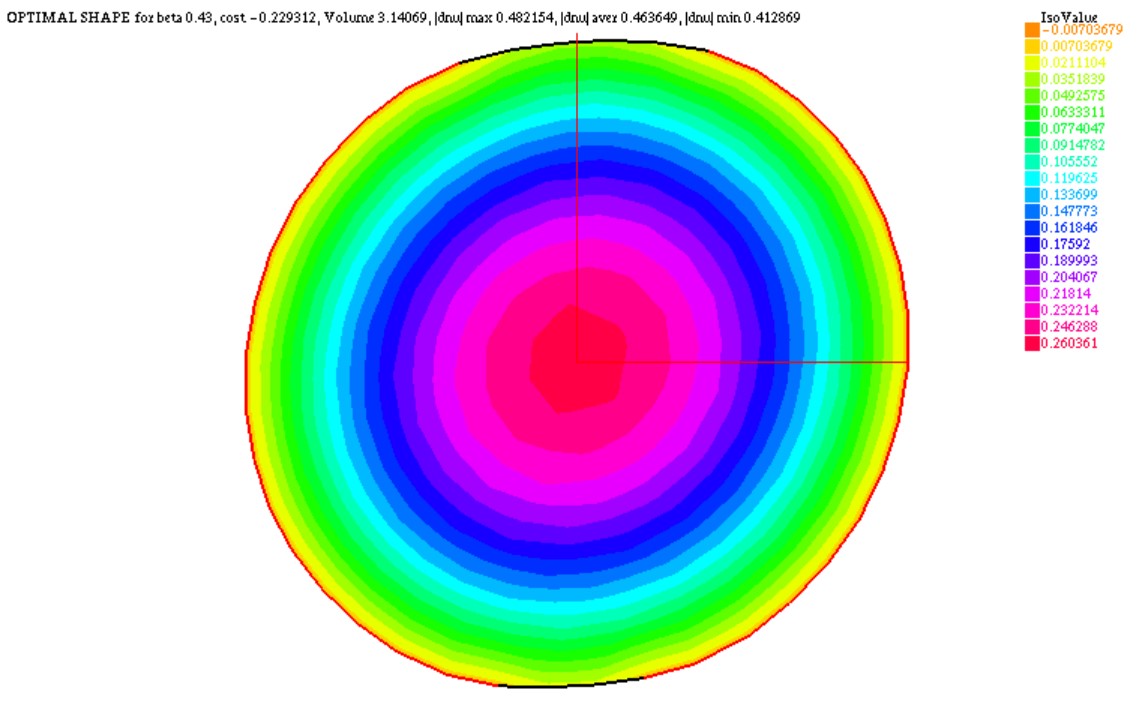}
    \includegraphics[scale=0.45, trim = 4.1cm 0.1cm 3.7cm 0.6cm, clip]{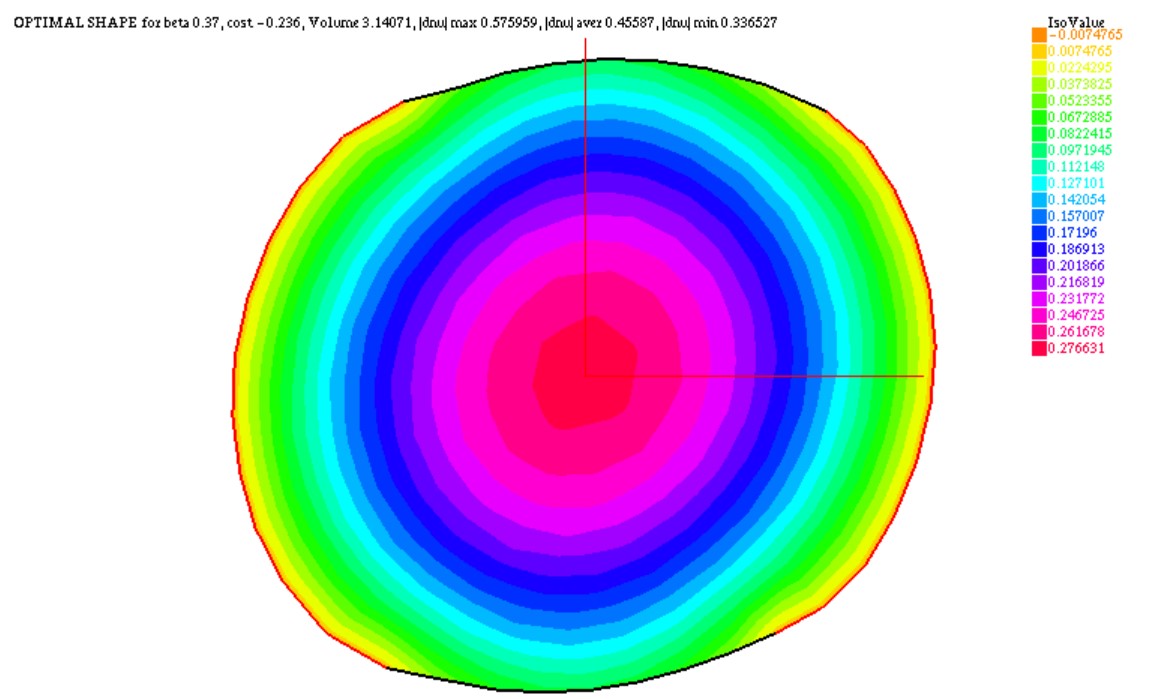} \hspace{2cm} 
    \includegraphics[scale=0.31, trim = 5.1cm 0.1cm 4.1cm 0.7cm, clip]{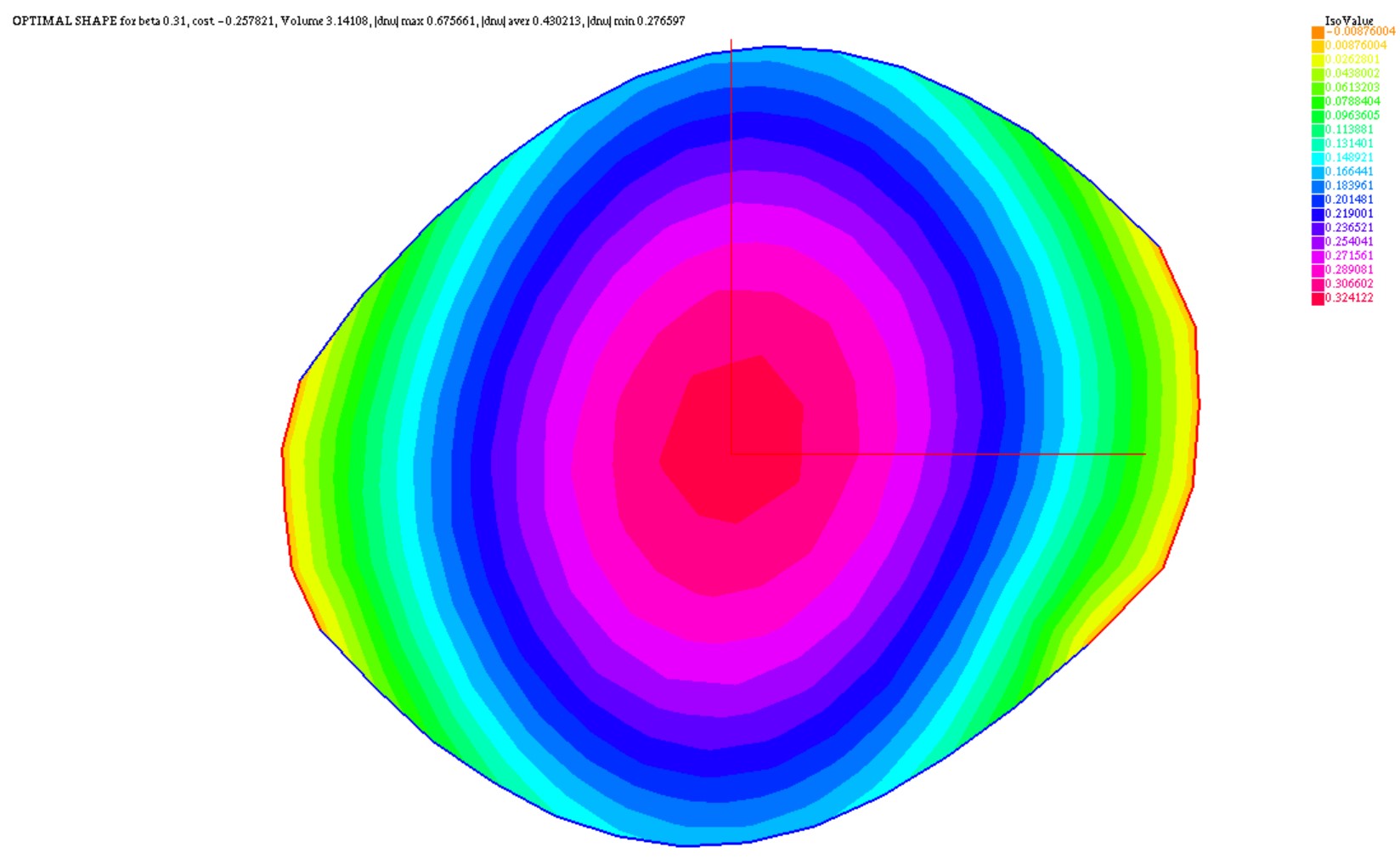}
    \caption{Shapes minimizing $\mathcal{J}$ under the volume constraint $\lambda=\pi$. From top-left to bottom-right,~$\beta=0.46,0.43,0.37,0.31$. The red boundary shows where $u=0$ and the black/blue boundary shows where~$|\partial_{\nn}u|=g_{\beta}$.}
    \label{optishapeTresca}
\end{figure}

\smallskip

\item Figure~\ref{optishapeTrescaNeumann} shows on the left the shapes which solve Problem~\eqref{shapeOptim} for $\beta=0.28, 0.1, 0.01$. Here the normal derivative of the solution $u$ to the scalar Tresca friction problem~\eqref{Trescaproblem222} reaches the friction threshold $g_\beta$ on the entire boundary. Moreover we can notice that these shapes are very close to the ones (presented on the right of Figure~\ref{optishapeTrescaNeumann}) that minimize~$\mathcal{J}_\mathrm{N}$ with~$g=g_\beta$ (see Remark~\ref{energyfunctionalNeu}) under the same volume constraint $\lambda=\pi$. Indeed, for these values of~$\beta$, one can check numerically that the solution $w_{\mathrm{N},\Omega}$ to the Neumann problem~\eqref{Neumannproblem222} with~$g=g_\beta$ satisfies $w_{\mathrm{N},\Omega}>0$ on $\Gamma$, and thus is also the solution to the scalar Tresca friction problem~\eqref{Trescaproblem222}. One deduces from Remark~\ref{energyfunctionalNeu} that the shape gradient of $\mathcal{J}$ and the one of $\mathcal{J}_\mathrm{N}$ coincide. Therefore, since the shape minimizing the Neumann energy functional~$\mathcal{J}_{\mathrm{N}}$ under the volume constraint $\lambda=\pi$ is a critical shape of the \textit{augmented} Neumann energy functional, it is also a critical shape of the \textit{augmented} Tresca energy functional.
\begin{figure}[h!]
    \centering
    \includegraphics[scale=0.44, trim = 4.3cm -0.4cm 4.2cm 0.6cm, clip]{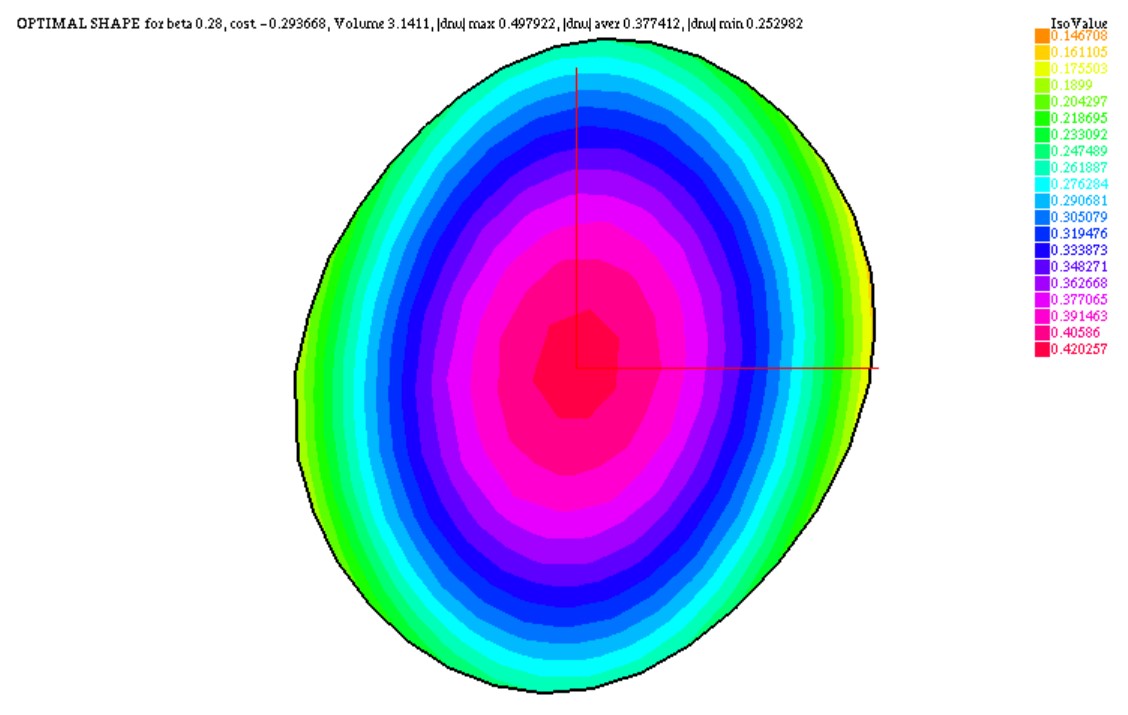}  \hspace{2cm} 
    \includegraphics[scale=0.44, trim = 4.3cm -0.4cm 4.2cm 0.6cm, clip]{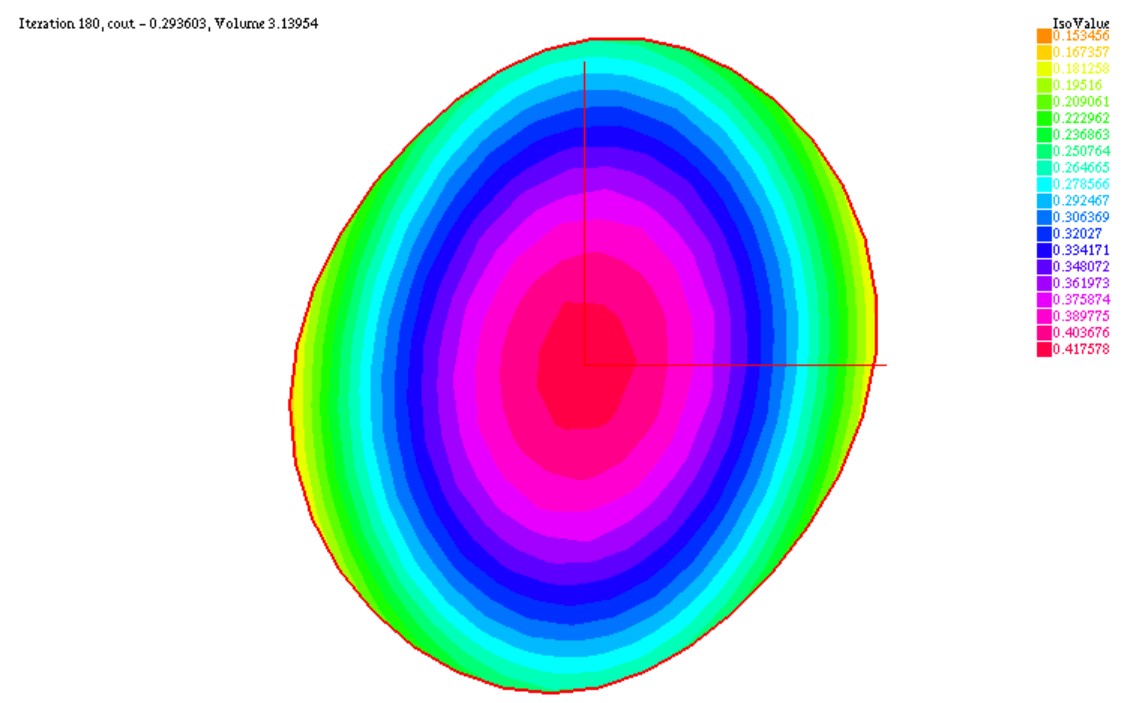}
    \includegraphics[scale=0.45, trim = 4.3cm -0.4cm 4.2cm 0.6cm, clip]{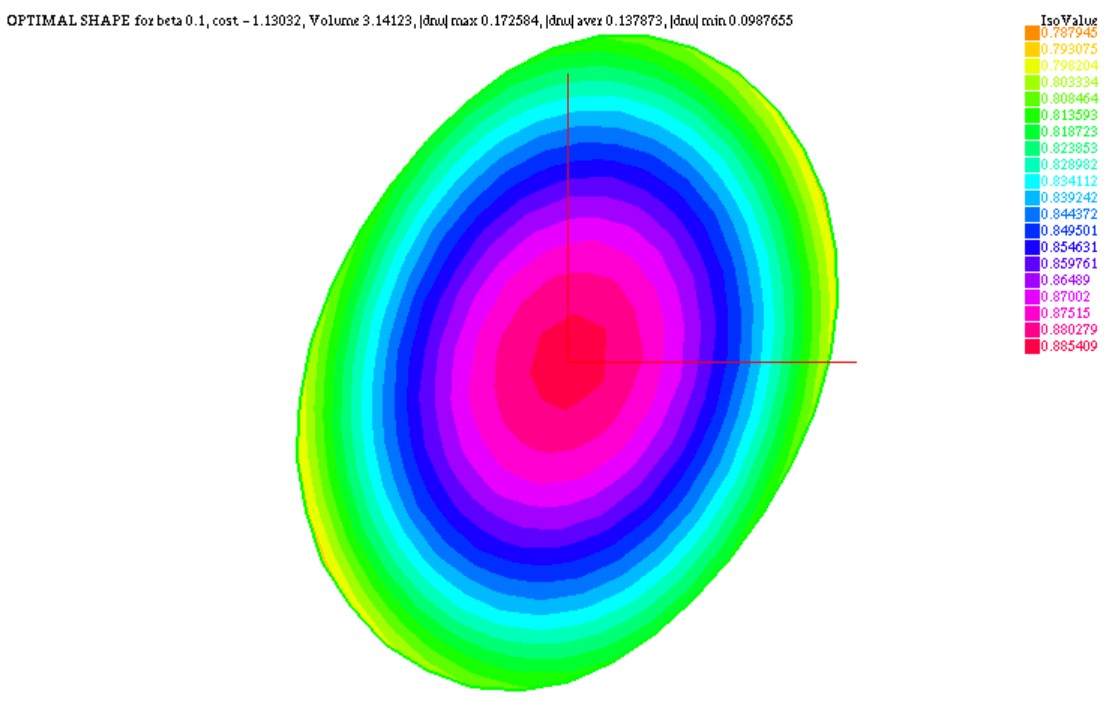}  \hspace{2cm} 
    \includegraphics[scale=0.45, trim = 4.3cm -0.4cm 4.2cm 0.5cm, clip]{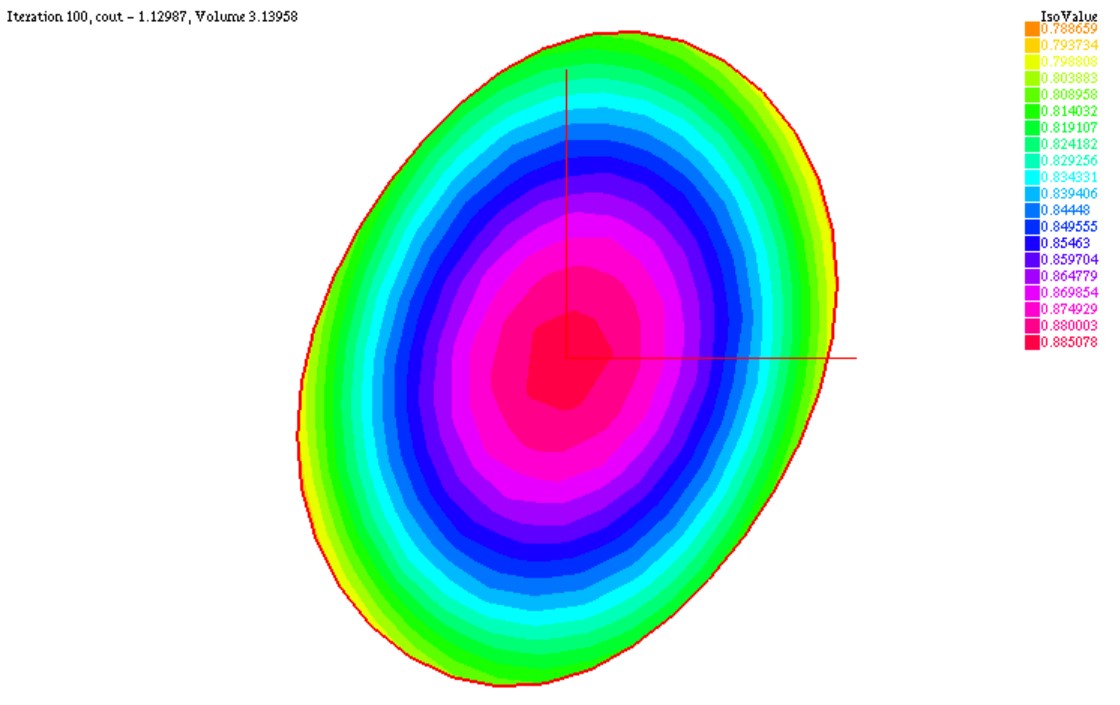}
    \includegraphics[scale=0.44, trim = 4.3cm 0.1cm 3.7cm 0.6cm, clip]{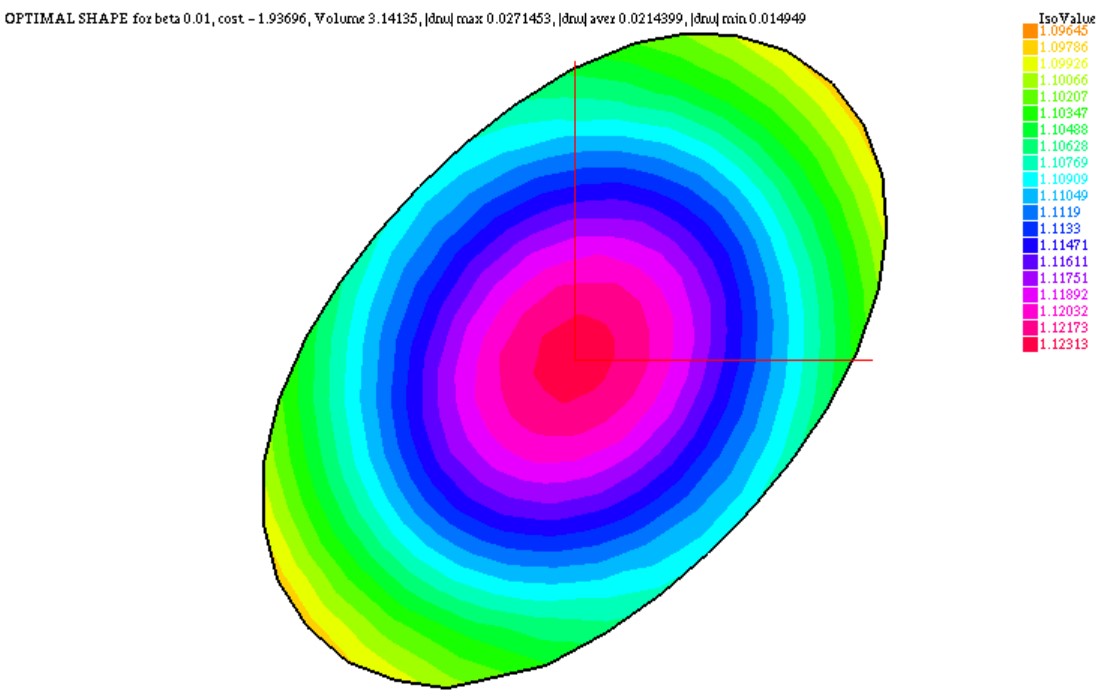}  \hspace{2cm} 
    \includegraphics[scale=0.44, trim = 4.3cm 0.1cm 4.2cm 0.6cm, clip]{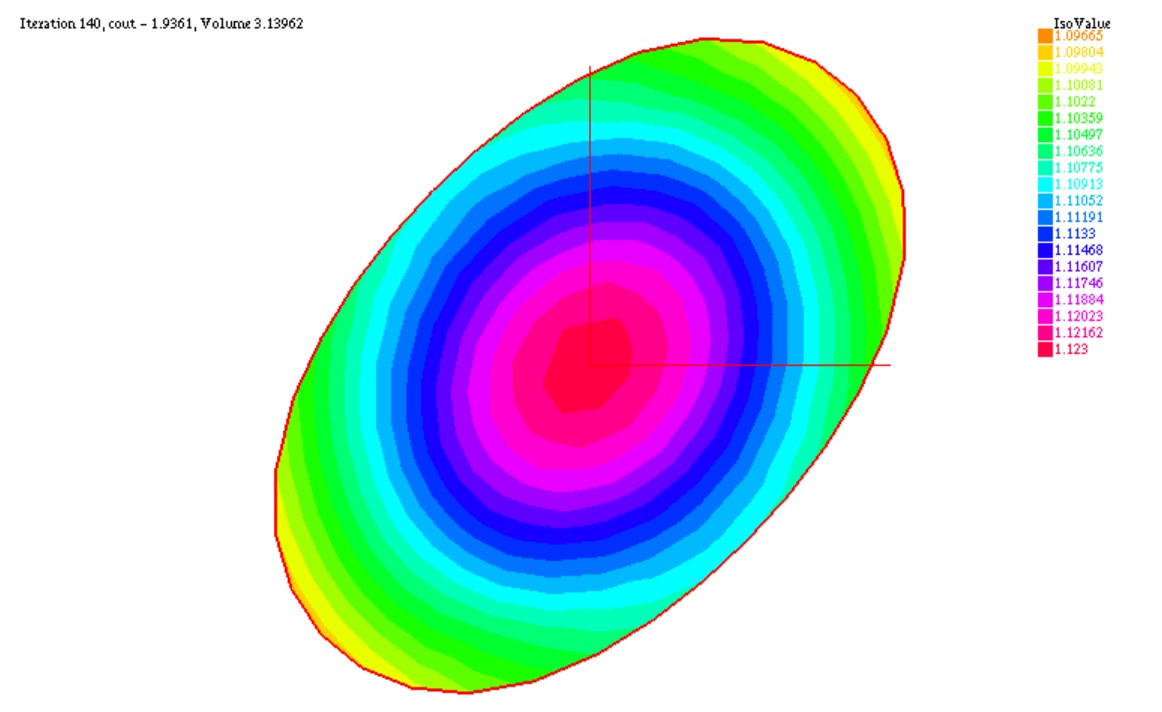}
    \caption{Shapes minimizing $\mathcal{J}$ (left) and $\mathcal{J}_{\mathrm{D}}$ (right), under the volume constraint $\lambda=\pi$. From top to bottom, $\beta=0.28,0.1,0.01$. }
    \label{optishapeTrescaNeumann}
\end{figure}
\end{itemize}

\smallskip

For more details and an animated illustration, we would like to suggest to the reader to watch the video \url{https://youtu.be/_MufZx3zsew} presenting all numerical results we obtained for different values of $\beta$ from~$0.7$ to $0.01$.

To conclude this paper, we would like to bring to the attention of the reader that, in the above numerical simulations, it seems that there is a kind of transition from optimal shapes associated with the Neumann energy functional to optimal shapes associated with the
Dirichlet energy functional. This transition is carried out by optimal shapes associated with the Tresca energy functional, continuously with respect to the friction threshold (precisely with respect to the parameter~$\beta$). However, we do not have a proof of such a highly nontrivial result. This may constitute an interesting topic for future investigations.

\bibliographystyle{abbrv}
\bibliography{biblio}

\end{document}